%% file: main.tex
\documentclass[letterpaper, 10 pt, conference]{ieeeconf}  

\IEEEoverridecommandlockouts                     

\usepackage{graphics} 
\usepackage{epsfig} 
\usepackage{mathptmx} 
\usepackage{times} 
\usepackage{amsmath} 
\usepackage{amssymb}  
\usepackage{textgreek,mathtools}
\usepackage{graphicx,subcaption}
\usepackage[colorlinks=true, allcolors=blue]{hyperref}
\usepackage{upgreek}
\usepackage{amsfonts}

\usepackage{amsthm}
\usepackage{mathrsfs}
\usepackage{fontenc}
\usepackage{empheq}
\usepackage{enumerate}
\usepackage{comment}
\usepackage{algorithm}
\usepackage{algorithmic}
\usepackage{bbm}
\usepackage{cancel}
\usepackage{tcolorbox}
\mathtoolsset{showonlyrefs}

\usepackage{cite}

\usepackage{xfrac}

\newcommand{\N}{\mathbb{N}}

\usepackage{etoolbox}

\input{Macros}

\title{\LARGE \bf Deception in Linear-Quadratic Control
}

\author{Yerin Kim$^{1}$, Haosheng Zhou$^{2}$, Alexander Benvenuti$^{1}$, Ruimeng Hu$^{2}$, Matthew Hale$^{1}$
\thanks{
$^{1}$School of  Electrical and Computer Engineering, Georgia Institute of Technology, Atlanta, GA, USA.
Emails: \texttt{\{yerinkim, abenvenuti3,matthale\}@gatech.edu}.
}
\thanks{$^{2}$Department of Mathematics, and Department of
Statistics and Applied Probability, University of California, Santa Barbara,
CA, USA.
Email: \texttt{\{hzhou593, rhu\}@ucsb.edu}.}
\thanks{
This work was partially supported by ONR
under grant N00014-24-1-2432 and the NSF Graduate
Research Fellowship under grant DGE-2039655. Any opinions, findings and conclusions or recommendations expressed in this material are those of the authors and do not necessarily reflect the views of sponsoring agencies.
}
}

\begin{document}

\maketitle
\thispagestyle{empty}
\pagestyle{empty}

\begin{abstract}
Systems operating in adversarial environments may inadvertently leak sensitive information to adversaries. 
To address this challenge, we revisit the linear-quadratic control framework and introduce deception to actively mislead adversaries. 
Specifically, we consider a blue-team agent, observed by a red-team agent, that seeks to minimize a quadratic cost while introducing perturbations to its trajectories over time.
These perturbations are designed to corrupt the red team's observations and, consequently, any downstream inferences, while remaining undetected by a red team using sequential hypothesis testing.
We implement this idea by augmenting the blue team's quadratic cost with a likelihood ratio statistic.
Under this augmented control problem, we derive a semi-explicit solution for the optimal deceptive control law and establish corresponding well-posedness results.
In addition, we provide both numerical approximations and analytical bounds for the probability 
that the red team detects
the blue team's deceptive strategies. 
Numerical results demonstrate the effectiveness of the proposed framework in deceiving the red team while remaining undetected with probability near~$1$.
\end{abstract}


\input{1-Intro-E}

\input{2-Prelims}

\input{3-Methodology-E}

\input{4-Results-E}

\input{5-Conclusion}






\bibliographystyle{IEEEtran}
\bibliography{ref}

\input{6-Appendix-E}

\end{document}

%% file: Macros.tex
\theoremstyle{definition}
\newtheorem{definition}{\protect\definitionname}
\newtheorem{assumption}{\protect\assumptionname}
\newtheorem{lemma}{\protect\lemmaname}

\newtheorem{remark}{Remark}

\theoremstyle{plain}
\newtheorem{theorem}{\protect\theoremname}
\newtheorem{problem}{Problem}

\newtheorem{corollary}{Corollary}

\DeclareMathOperator{\tr}{Tr}

\newcommand{\pushright}[1]{\ifmeasuring@#1\else\omit\hfill$\displaystyle#1$\fi\ignorespaces}

\makeatother

\providecommand{\assumptionname}{Assumption}
\providecommand{\definitionname}{Definition}
\providecommand{\lemmaname}{Lemma}
\providecommand{\theoremname}{Theorem}
\providecommand{\propositionname}{Proposition}
\newcommand{\norm}[1]{\left\lVert#1\right\rVert}

\newcommand{\trace}[1]{\tr\left(#1\right)}

\newcommand{\xv}{\mathbf{x}}
\newcommand{\uv}{\mathbf{u}}
\newcommand{\wv}{\mathbf{w}}
\newcommand{\yv}{\mathbf{y}}
\newcommand{\E}{\mathbb{E}}
\newcommand{\EPS}{\varepsilon}
\newcommand{\R}{\mathbb{R}}
\newcommand{\upperthreshold}{U}
\newcommand{\lowerthreshold}{L}
\newcommand{\DeceptionMeasure}{D}

\newboolean{show_comments}
\setboolean{show_comments}{true} 
\ifthenelse{\boolean{show_comments}}{
    \newcommand{\Alex}[1]{\textcolor{blue}{Alex:~#1}} 
    \newcommand{\mh}[1]{\textcolor{red}{MH:~#1}} 
} {
    \newcommand{\Alex}[1]{}
    \newcommand{\mh}[1]{}
    \newcommand{\AuthorC}[1]{}
}

%% file: 1-Intro-E.tex
\section{Introduction}\label{sec:intro}
Linear-quadratic (LQ) control has been used in systems such as UAVs~\cite{martins2019linear, dharmawan2017model}, smart grids~\cite{keshtkar2014proposing}, and satellites~\cite{ke2017study}, among 
many others~\cite{jin2023research, ebrahim2010application, van2002industrial, rakhmatillaev2025integrative, kumar2023review}. 
However, conventional LQ control may leak information, such as cost functions~\cite{zhang2019inverse} or target states~\cite{li2026inverse}, to adversaries observing the system, 
which may enable adversaries to take targeted counteractions~\cite{zhang2014optimal,zhou2020optimal}.
To prevent such countermeasures, it is desirable for systems to mislead adversaries about their intentions while simultaneously minimizing their cost functions.

In this work, we study deceptive LQ control in settings where a blue team agent is observed by a red
team agent. The blue team pursues two simultaneous objectives:
(i) to execute a primary task, encoded as the minimization of a given quadratic cost functional, and (ii) to perturb its control inputs (and hence the resulting state trajectories) to corrupt the red team's observations and mislead the adversary about its intentions.
We model the red team as not only observing the blue team, but also attempting
to infer whether the blue team is behaving deceptively. 

Under this setting, if the blue team introduces large perturbations, it may effectively corrupt the red team's observations.
However, doing so may lead to two undesirable consequences: first, it may significantly degrade performance in minimizing the cost functional, and second, it makes the deception easily detectable by the red team.
Motivated by these trade-offs, the blue team seeks to introduce deceptive perturbations at a level that remains undetectable (i.e., the red team believes that no such deception is taking place), which we refer to as the ``stealthy deception'' problem.  

To model the detection mechanism, we consider a red team that employs sequential hypothesis testing, a powerful and general-purpose tool that typically requires fewer observations than comparable techniques~\cite{doerks2026relative}. 
We consider a sophisticated red team that has knowledge about the blue team's deceptive control strategy through external information sources,
enabling a worst-case analysis from the blue team's perspective.

\subsection{Summary of Contributions}
Our main contributions are summarized as follows:
\begin{enumerate}
    \item We propose a discrete-time deceptive LQ control framework that integrates a likelihood ratio statistic from sequential hypothesis testing into the cost functional, scaled by an intensity parameter.  This augmented control problem preserves the quadratic structure of the cost and is analytically tractable. Using standard control techniques, we derive a semi-explicit characterization of the optimal deceptive controller and establish an associated well-posedness result. Compared to continuous-time formulations \cite{zhou2025integrating,zhou2025adversarial}, the proposed discrete-time model provides a more realistic depiction of strategic interactions and admits exact solutions in terms of recursions rather than ordinary differential equations, while the corresponding well-posedness analysis is technically more challenging. 
    \item We study the stealthiness of the proposed deceptive control strategy under sequential hypothesis testing (SHT) as the inference mechanism. To the best of our knowledge, this setting has not been studied in the existing literature. We derive analytical bounds on the detection probability using concentration inequalities and develop sampling-based methods to approximate it, providing
guidance for selecting the deception intensity to balance
the deception strength against the risk of detection.
    \item 
    We conduct numerical experiments to solve for the deception intensity, under which deceptive perturbations are present in the trajectories while remaining undetected by the adversary.
    This demonstrates the effectiveness and practical value of the proposed framework.
\end{enumerate}

\subsection{Related Works}\label{subsec:related_works}

Adversarial decision-making has been studied in the context of both privacy and deception.
Differential privacy has been widely studied to protect sensitive information~\cite{hale2018privacy, benvenuti2024guaranteed}, but it does not provide active deception of adversaries, which is the main focus in this work. 
A key strength of differential privacy is that it provides general-purpose privacy protections, rather than countering a specific form of inference by an observer. 
We pursue a similar objective by designing deception that corrupts an adversary's downstream inferences in a general manner, rather than seeking to counter a particular one.
However, our approach differs fundamentally: differential privacy seeks to induce uncertainty in an observer, whereas the deception mechanism we develop aims to make the adversary confidently draw the incorrect conclusion that no deception is present. 
As a result, the adversary is led to trust corrupted information when conducting downstream inferences.

Deception has been studied in the context of supervisory control~\cite{karabag2021deception, karabag2022exploiting} and adversarial control~\cite{fotiadis2025deception, kim2024defining}.
Our work differs from~\cite{karabag2021deception} and~\cite{karabag2022exploiting} in that we consider adversarial settings without a reference policy to follow. It also differs from~\cite{fotiadis2025deception} and~\cite{kim2024defining} by considering a setting in which the adversary is aware that the blue team may act deceptively and knows the form such deception may take.

In the context of adversarial interaction, sequential hypothesis testing (SHT) 
has previously been used to detect adversarial attacks~\cite{ho2011fast, vamsi2014sybil}.
In contrast, we use SHT both to construct deception-aware cost functions and as an inference mechanism employed by the red team.
Integrating SHT into deceptive control was studied in~\cite{zhou2025integrating,zhou2025adversarial} in continuous-time LQ control. 
Our work differs from~\cite{zhou2025integrating,zhou2025adversarial} in two main aspects.
First, we adopt a discrete-time LQ setting in which the red team fully observes the state of the blue team, whereas~\cite{zhou2025integrating,zhou2025adversarial} rely on partial observability for deception.
Second, we explicitly study the stealthiness of deception, which was not addressed in~\cite{zhou2025integrating,zhou2025adversarial}.

\subsection{Notations}\label{subsec:notation}
We use~$\N$ to denote the positive integers and~$\R$ to denote the reals. 
For~$N\in\mathbb{N}$, we define~$[N] := \{0, 1, \ldots, N\}$. 
We use~$I_n$ for the~$n\times n$ identity matrix and~\(\mathbf{0}_{m\times n}\) for the~$m\times n$ zero matrix.
We write~\(I\) or~\(0\) when the dimension is clear from context.
For a symmetric matrix~$M$, we use $\norm{\xv}_M^2:= \xv^TM\xv$.
Matrix inequalities \(\succ\) (resp. \(\succeq\)) are understood in the positive definite (resp. positive semi-definite) sense.
The symbol $\mathbb{I}_{\phi}$ denotes the indicator function of
a mathematical statement~$\phi$; $\mathbb{I}_{\phi} = 1$ if~$\phi$ is true
and~$\mathbb{I}_{\phi} = 0$ if~$\phi$ is false. 
We denote by~\(X\sim \mathcal{N}(\mu, \Sigma) \) an \(\R^n\)-valued random variable~$X$ \
following the Gaussian distribution with mean~$\mu \in \mathbb{R}^n$ and covariance~$\Sigma \in \mathbb{R}^{n \times n}$.
We use~$\mathcal{Q}:\R\rightarrow\R$ to denote the tail integral of the standard normal distribution, 
i.e., for~$X\sim\mathcal{N}(0,1)$ we have~$\mathcal{Q}(x) := \mathbb{P}(X>x) = \frac{1}{\sqrt{2\pi}}\int_x^\infty \exp(-\frac{u^2}{2}) du$.
We use $a \wedge b$ for $\min\{a,b\}$ and $a \vee b$ for $\max\{a,b\}$.
The expressions~$\trace{M}$, \(\det(M)\), and \(M^s := (M+M^T)/2\) denote the trace, determinant, and symmetric part of a square matrix~$M$.
The operator~$\mathrm{Concat}(\xv_1, \dots, \xv_k) := (\xv_1^T, \ldots, \xv_k^T)^T$ concatenates the vectors~$\xv_1$ through~$\xv_k$, and \(\mathrm{diag}(A_1,\ldots,A_k)\) returns a block matrix with diagonal blocks \(A_1\) through \(A_k\).

%% file: 2-Prelims.tex
\section{Preliminaries and Problem Statements}\label{sec:prelims}
In this section, we present background material and formulate the main problems studied in the rest of the paper.

\subsection{Sequential Probability Ratio Test}
A sequential hypothesis test (SHT) uses sequential observations to test the null hypothesis~$\mathsf{H}_0$ against the alternative~$\mathsf{H}_1$. 
For SHT, the number of observations 
required is not predetermined, and the outcome at each time $t$
is one of the following: (i) accepts~$\mathsf{H}_0$,
(ii) rejects~$\mathsf{H}_0$, or (iii) continues by taking another observation. 
One specific instance of SHT is the sequential probability ratio test. 

\begin{definition}[Sequential Probability Ratio Test (SPRT)~\cite{wald1992sequential}]\label{def:sprt}
    Let~$P_{0m} := \mathbb{P}_{\mathsf{H}_0}(\{x_i\}_{i\in[m]})$ and~$P_{1m} := \mathbb{P}_{\mathsf{H}_1}(\{x_i\}_{i\in[m]})$ be the probabilities of observing the sequence~$\{x_i\}_{i\in[m]}$ under
    hypotheses~$\mathsf{H}_0$ and~$\mathsf{H}_1$, respectively, where $\mathbb{P}_{\mathsf{H}_0}(\cdot)$ and $\mathbb{P}_{\mathsf{H}_1}(\cdot)$ denote probabilities under $\mathsf{H}_0$ and $\mathsf{H}_1$. 
    The \emph{sequential probability ratio test} 
    rejects~$\mathsf{H}_0$ if~$P_{1m}/P_{0m}\geq \upperthreshold$, accepts~$\mathsf{H}_0$ if~$P_{1m}/P_{0m}\leq \lowerthreshold$, and takes 
    an additional observation if~$\lowerthreshold<P_{1m}/P_{0m}<\upperthreshold$, where~$0<\lowerthreshold<\upperthreshold$ are fixed constants and the ratio~$P_{1m}/P_{0m}$ is calculated at each time step~\(m\).
\end{definition}

SPRT is the most powerful sequential test for simple \textit{vs.} simple hypotheses and minimizes the expected number of observations needed to reach a decision~\cite{wald1992sequential}.
By Lemma~\ref{lem:sprt_parameters}, the thresholds $\upperthreshold$ and $\lowerthreshold$ can be chosen to satisfy prescribed error probability requirements. 

\begin{lemma}[SPRT Thresholds \cite{wald1992sequential}] 
\label{lem:sprt_parameters}
    In SHT, a ``type I error'' refers to rejecting $\mathsf{H}_0$ when~$\mathsf{H}_0$ is true, and 
    a ``type II error'' refers to accepting~$\mathsf{H}_0$ when~$\mathsf{H}_1$ is true.
    Suppose one sets
    \begin{equation}
        \upperthreshold = \frac{1-b}{a} \quad \textnormal{ and } \quad \lowerthreshold = \frac{b}{1-a}
    \end{equation}
    for~$a,b\in(0,\sfrac{1}{2})$
    and implements the SPRT according to the decision rule in Definition~\ref{def:sprt}. 
    Then 
    the probabilities of type I and II errors are respectively bounded by~$a$ and~$b$, i.e.,~$\mathbb{P}_{\mathsf{H}_0}(\text{Reject }\mathsf{H}_0)\leq a$ and~$\mathbb{P}_{\mathsf{H}_1}(\text{Accept }\mathsf{H}_0)\leq b$.
\end{lemma}

\subsection{Problem Statements}
In this work, we model deception in a red team versus blue team setting, where the blue team tries to perturb its control inputs to corrupt the red team's observations, thereby concealing its intentions.
The red team, in turn, seeks to detect such deception by attempting to detect any perturbations introduced by the blue team. 
We assume that the blue team's dynamics are known to both teams, and that the red team can observe full state trajectories of the blue team.
We suppose the red team
has complete knowledge of the deception strategy implemented by the blue team, which provides
a worst-case analysis from the blue team's perspective. 
By tuning the deception intensity, the blue team can either enlarge or shrink the detectability of its deception.

Under this setting, we aim to solve the following problems:
\begin{problem}\label{prob:control_law}
    Formulate a deceptive cost for the blue team,
    and compute the resulting optimal control law that fulfills the blue team's primary task
    while introducing deceptive perturbations. 
\end{problem}
\begin{problem}\label{prob:stealthy_misdirection}
    Compute the deception intensity under which the induced deceptive perturbations remain undetected by the red team.     
\end{problem}
\begin{problem}\label{prob:empirical_verification}
    Empirically evaluate the performance of the blue team's deceptive control
    in terms of the incurred quadratic cost and the probability of detection by the red team. 
\end{problem}

%% file: 3-Methodology-E.tex
\section{Blue team control design}\label{sec:control_design}

In this section, we address Problem~\ref{prob:control_law}.
We first introduce the blue team's dynamics and costs
associated with the primary task, and then augment this cost with a deception term and derive the 
resulting optimal deceptive controller.

\subsection{Primary Task}\label{subsec:primary_task}
Fix a time horizon~$T \in \mathbb{N}$, and
consider a probability space~$(\Omega, \mathcal{F}, \mathbb{P})$ 
supporting i.i.d. Gaussian random variables~$Z_t$ and~$Y_t$, {where~$Z_t\sim\mathcal{N}(0, \Sigma^2_z)$, $Y_t\sim\mathcal{N}(0, \Sigma^2_y)$,
with symmetric~$\Sigma_z^2 \succ 0$ and~$\Sigma_y^2 \succ 0$. 
The state dynamics of the blue team are governed by Markovian controls \(\alpha,\beta\), and are given by
\begin{equation}\label{eq:state_dynamics}
    \begin{aligned}
    \mathbf{v}_{t+1} &= \alpha_t + \mathbf{v}_t + Z_t,\ \mathbf{v}_0\in\R^n,\\
    \mathbf{p}_{t+1} &= (\mathbf{v}_t+\beta_t) +\mathbf{p}_t + Y_t,\ \mathbf{p}_0\in\R^n,
    \end{aligned}
\end{equation}
where~$\mathbf{v}_t, \mathbf{p}_t, \alpha_t, \beta_t$ take values in $\R^n$. Equivalently,
\begin{align}\label{eq:matrix_form_state_dynamics}
    &\xv_{t+1} = A\xv_t + B\uv_t + \wv_t,\ \text{where }\xv_t := \mathrm{Concat}(\mathbf{v}_t,\mathbf{p}_t),\\
    &\uv_t := \begin{bmatrix} \alpha_t\\ \beta_t\end{bmatrix},\ \wv_t := \begin{bmatrix} Z_t\\ Y_t\end{bmatrix},\ A := \begin{bmatrix}
    I_n&\mathbf{0}_{n\times n}\\I_n&I_n
    \end{bmatrix},\ B := I_{2n},
\end{align}
with a given deterministic initial state~$\xv_0 \in\R^{2n}$.
As its primary task, the blue team minimizes the expected cost
\begin{equation}\label{eq:primary_cost}
    J^{\mathrm{primary}}(\alpha, \beta) = \E\Big[\sum_{t=0}^{T-1} r_t(\mathbf{v}_t, \mathbf{p}_t, \alpha_t, \beta_t) + g(\mathbf{v}_T, \mathbf{p}_T)\Big],
\end{equation}
where the running and terminal costs are given by
\begin{equation}
\label{eq:primary_task}
    \begin{aligned}      
    r_t(\mathbf{v},\mathbf{p},\alpha,\beta) & = \tfrac12{\alpha}^TR_\alpha\alpha + \tfrac12\beta^TR_\beta\beta 
     + \tfrac12(\mathbf{v} - \bar{\mathbf{v}}_t)^TR_v(\mathbf{v} - \bar{\mathbf{v}}_t),\\
    g(\mathbf{v},\mathbf{p}) &= \tfrac12(\mathbf{v}-\bar{\mathbf{v}}_T)^T T_v(\mathbf{v}-\bar{\mathbf{v}}_T).
    \end{aligned}
\end{equation}
Here, $R_\alpha, R_\beta, R_v, T_v \in\R^{n\times n}$ are symmetric positive 
definite, and the reference trajectory $\bar{\mathbf{v}}_t \in \mathbb{R}^n$ is given for any $t \in [T]$.

\begin{remark}
    Since the running and terminal costs do not directly depend on $\mathbf{p}_t$, 
    any choice of $\beta \not\equiv 0$ increases the blue team's expected cost. 
    Hence, the optimal control for the primary task is given by $\beta^* \equiv 0$.
    In what follows, we interpret the control $\beta$ as being used to implement deception; consequently, any introduction of deception (i.e., any deviation from \(\beta\equiv 0\)) increases $J^{\mathrm{primary}}$. 
\end{remark}

\subsection{SHT Formulation}

The blue team has pre-determined 
deception patterns that it seeks to use, which are given by
the sequences of~\(n\times n\) matrices~$F^b := \{F^b_t\}_{t \in [T-1]}$ and~$F^c := \{F^c_t\}_{t \in [T-1]}$, and the \(\R^n\)-valued vectors~$\mathbf{f}^d := \{\mathbf{f}^d_t\}_{t \in [T-1]}$.
Then, the blue team's linear deceptive control law for~$\beta$ has the form
\begin{equation}\label{eq:purely_deceptive_beta}
    \beta_t = \Theta^b_t \mathbf{v}_t + \Theta^c_t \mathbf{p}_t + \mathbf{\theta}^d_t \textnormal{ for all } t \in [T-1],
\end{equation}
and 
its goal is to use this controller (or a weighted version thereof)
in order to deceive the red team. 
The blue team models the red team as using sequential
hypothesis testing with two hypotheses, namely
\begin{subequations} \label{eq:sht_hypothesis}    
\begin{align}
        \mathsf{H}_0&: \Theta^b_t = \Theta^c_t = \mathbf{0}_{n \times n}, \mathbf{\theta}^d_t = \mathbf{0}_{n \times 1} \textnormal{ for all } t \in [T-1],\label{eq:h0def} \\
        \mathsf{H}_1&: \Theta^b_t = F^b_t, \Theta^c_t = F^c_t, \mathbf{\theta}^d_t = \mathbf{f}^d_t \textnormal{ for all } t \in [T-1]. \label{eq:h1def} 
\end{align}
\end{subequations}

The null hypothesis $\mathsf{H}_0$ corresponds to the belief that the blue team is not introducing deception, while the alternative hypothesis $\mathsf{H}_1$ reflects the blue team's adoption of the deceptive control law \(F^b_t \mathbf{v}_t + F^c_t \mathbf{p}_t + \mathbf{f}^d_t\).
Consequently, if the red team accepts~$\mathsf{H}_0$, it regards the blue team's trajectories as credible for downstream analysis, such as cost inference. 
In contrast, if the red team accepts~$\mathsf{H}_1$, it confirms that the blue team's trajectories are subject to deceptive perturbations and are therefore not credible for further analysis.

Since~\(\xv\) remains a Markov chain under Markovian controls~\(\alpha\) and~\(\beta\), by Definition~\ref{def:sprt}, the SPRT statistic at time~$t$ takes the form 
\begin{equation}\label{eq:L_T}
    L_t(\xv) := \frac{\mathbb{P}_{\mathsf{H}_1} (\{\xv_k\}_{k\in[t]})}{\mathbb{P}_{\mathsf{H}_0} (\{\xv_k\}_{k\in[t]})}= \prod_{k=0}^{t-1}\frac{\mathbb{P}_{\mathsf{H}_1} (\xv_{k+1}\mid\xv_k)}{\mathbb{P}_{\mathsf{H}_0} (\xv_{k+1}\mid \xv_k)}.
\end{equation}
Intuitively, a larger (resp. smaller) $L_t(\xv)$ indicates that the observed trajectory 
$\{\xv_k\}_{k\in[t]}$ is more likely to have been generated by the control~$\beta$ associated with~\(\mathsf{H}_1\) in~\eqref{eq:h1def} (resp.~\(\mathsf{H}_0\) in~\eqref{eq:h0def}).
Using the state dynamics~\eqref{eq:state_dynamics}, we get
\begin{multline}\label{eq:P_H_0_at_t}
    \mathbb{P}_{\mathsf{H}_0}(\xv_{k+1}\mid \xv_k) = (2\pi)^{-n} [\det (\Sigma^2_z) \,\det(\Sigma^2_y)]^{-\frac12} \\ \exp\Big(-\frac12\norm{\mathbf{v}_{k+1} - \mathbf{v}_{k} - \alpha_{k}}_{(\Sigma_z^2)^{-1}}^2 \\ - \frac12 \norm{\mathbf{p}_{k+1} - \mathbf{v}_{k} - \mathbf{p}_{k}}_{(\Sigma_y^2)^{-1}}^2\Big),
\end{multline}
\vspace{-2em}
\begin{multline}\label{eq:P_H_1_at_t}
    \mathbb{P}_{\mathsf{H}_1}(\xv_{k+1}\mid \xv_k) =(2\pi)^{-n} [\det (\Sigma^2_z) \,\det(\Sigma^2_y)]^{-\frac12}\\ \exp\Big(-\frac12\norm{\mathbf{v}_{k+1} - \mathbf{v}_{k} - \alpha_{k}}_{(\Sigma_z^2)^{-1}}^2\\ - \frac12 \big\|{\mathbf{p}_{k+1} - \mathbf{v}_{k} - \mathbf{p}_{k} -F^b_k \mathbf{v}_k - F^c_k \mathbf{p}_k- \mathbf{f}^d_k}\big\|_{(\Sigma_y^2)^{-1}}^2\Big).
\end{multline}
Plugging~\eqref{eq:P_H_0_at_t} and~\eqref{eq:P_H_1_at_t} into~\eqref{eq:L_T} and taking the logarithm yields
\begin{multline}\label{eq:log_L_T}
    \log L_t(\xv)  = \frac12\sum_{k=0}^{t-1} \big[  \norm{\mathbf{p}_{k+1} - \mathbf{v}_{k} - \mathbf{p}_{k}}_{(\Sigma_y^2)^{-1}}^2 \\ - \big\|{\mathbf{p}_{k+1} - \mathbf{v}_{k} - \mathbf{p}_{k} -F^b_k \mathbf{v}_k - F^c_k \mathbf{p}_k- \mathbf{f}^d_k}\big\|_{(\Sigma_y^2)^{-1}}^2\big]. 
\end{multline}

From~\eqref{eq:log_L_T}, \(\E\log L_t\) admits a representation as the expectation of a sum over all time steps.
However, the~$k^{th}$ summand depends on the future state~$\mathbf{p}_{k+1}$, and therefore cannot be directly incorporated into the running cost~\eqref{eq:primary_task} to give a Markovian control problem.
Nevertheless, using the dynamics in~\eqref{eq:state_dynamics},
we rewrite \(\mathbb{E}\log L_t\) as
\begin{multline}\label{eq:def_logLt}
    \mathbb{E}\log L_t(\xv)  = \frac12 \mathbb{E}\sum_{k=0}^{t-1} \big[ \norm{\beta_k + Y_k}_{(\Sigma_y^2)^{-1}}^2 \\ -\big\|{\beta_k + Y_k -F^b_k \mathbf{v}_k - F^c_k \mathbf{p}_k- \mathbf{f}^d_k}\big\|_{(\Sigma_y^2)^{-1}}^2\big]. 
\end{multline}
Since the control~$\beta_k$ and states~$\mathbf{v}_k$, $\mathbf{p}_k$ are independent of the centered Gaussian noise~\(Y_k\), 
we find 
\begin{align}
    \label{eqn:E_log_LT}
    &\mathbb{E}\log L_t(\xv)  = \frac12\mathbb{E}\sum_{k=0}^{t-1}  \big[ \norm{\beta_k}_{(\Sigma_y^2)^{-1}}^2 + \norm{Y_k}_{(\Sigma_y^2)^{-1}}^2 \\ 
    &\qquad - \big\|{\beta_k-F^b_k \mathbf{v}_k - F^c_k \mathbf{p}_k- \mathbf{f}^d_k}\big\|_{(\Sigma_y^2)^{-1}}^2- \norm{Y_k}_{(\Sigma_y^2)^{-1}}^2\big]\\
    &= \frac12\mathbb{E}\sum_{k=0}^{t-1}  \big[\norm{\beta_k}_{(\Sigma_y^2)^{-1}}^2 - \big\|{\beta_k-F^b_k \mathbf{v}_k - F^c_k \mathbf{p}_k- \mathbf{f}^d_k}\big\|_{(\Sigma_y^2)^{-1}}^2\big].
\end{align}
The~$k^{th}$ summand in this new expression only depends on the current states~$\mathbf{v}_k$, $\mathbf{p}_k$ and the control~\(\beta_k\), which facilitates its direct incorporation into the running cost~\eqref{eq:primary_task} when solving the deception model in Section~\ref{subsec:solving_misdirection_model}.

\subsection{Solving the Deception Model}\label{subsec:solving_misdirection_model}

Besides fulfilling the primary task, the blue team introduces deceptive perturbations by maximizing~$\E\log L_T$, thereby driving its control~$\beta_t$ toward the deceptive input~$F^b_t \mathbf{v}_t + F^c_t \mathbf{p}_t + \mathbf{f}^d_t$.
This yields a bi-objective optimization problem with the following cost to be minimized:
\begin{equation}\label{eq:blue_team_cost}
    J^{\mathrm{blue}}(\uv) := J^{\mathrm{primary}}(\uv) - \lambda\mathbb{E}\log L_T,
\end{equation}
where $\lambda\geq 0$ is the intensity of deception.
By incorporating the summands in~\eqref{eqn:E_log_LT} into the running cost, we obtain a new Markovian control problem with the same state dynamics and terminal cost, but the modified running cost
\begin{multline}
    h_t(\mathbf{v},\mathbf{p},\alpha,\beta) = r_t(\mathbf{v},\mathbf{p},\alpha,\beta)\\
    + \frac{\lambda}{2} \big[\big\|{\beta-F^b_t \mathbf{v} - F^c_t \mathbf{p}- \mathbf{f}^d_t}\big\|_{(\Sigma_y^2)^{-1}}^2 - \norm{\beta}_{(\Sigma_y^2)^{-1}}^2 \big].
\end{multline}
For the convenience of subsequent calculations, we rewrite \(h_t\) as a function of \(\xv\) and \(\uv\) via
\begin{equation}\label{eq:modified_running_cost}
    h_t(\mathbf{x}, \mathbf{u}) = \frac12\mathbf{x}^TQ_t\mathbf{x} + \frac12\mathbf{u}^T R \mathbf{u} + \mathbf{x}^T N_t \mathbf{u} + \mathbf{q}_t^T\mathbf{x} + \mathbf{r}_t^T\mathbf{u} + d_t,
\end{equation}
where the coefficients are 
\begin{alignat}{2}\label{eq:matrices_for_new_cost}
    Q_t &:= \begin{bmatrix}
        R_v + \lambda(F_t^b)^T(\Sigma_y^2)^{-1}F_t^b & \lambda(F_t^b)^T(\Sigma_y^2)^{-1}F_t^c \\ \lambda(F_t^c)^T(\Sigma_y^2)^{-1}F_t^b & \lambda(F_t^c)^T(\Sigma_y^2)^{-1}F_t^c
    \end{bmatrix}, \,\, \\
    R& := \begin{bmatrix}
        R_\alpha & 0 \\ 0 & R_\beta
    \end{bmatrix}, \quad
    N_t := \begin{bmatrix}
        0& -\lambda(F_t^b)^T(\Sigma_y^2)^{-1} \\0&  -\lambda(F_t^c)^T(\Sigma_y^2)^{-1}
    \end{bmatrix},& \\
    \mathbf{q}_t& := \begin{bmatrix}
        -R_v\bar{\mathbf{v}}_t + \lambda(F_t^b)^T(\Sigma_y^2)^{-1}\mathbf{f}_t^d\\  \lambda(F_t^c)^T(\Sigma_y^2)^{-1}\mathbf{f}_t^d
    \end{bmatrix}, \
    \mathbf{r}_t := \begin{bmatrix}
        0 \\ -\lambda(\Sigma_y^2)^{-1}\mathbf{f}_t^d
    \end{bmatrix},  &\\
    d_t& := \frac12\bar{\mathbf{v}}_t^T{R_v}\bar{\mathbf{v}}_t+ \frac{\lambda}{2}(\mathbf{f}_t^d)^T(\Sigma_y^2)^{-1}\mathbf{f}_t^d.
\end{alignat}
Importantly, the augmented running cost \(h_t\) consists only of linear and quadratic terms in the states and controls, thus preserving the LQ structure as in \eqref{eq:state_dynamics}--\eqref{eq:primary_task}. 
Thanks to the analytical tractability of LQ control problems, we proceed to derive a semi-explicit solution for the optimal deceptive control using dynamic programming.

Define the associated value function~$V:[T]\times \mathbb{R}^{2n} \to \mathbb{R}$ as
\begin{equation}\label{eq:def_value_func}
    V_t(\mathbf{x}) := \inf_{\mathbf{u}}\,\mathbb{E} \Big[\sum_{k=t}^{T-1} h_k(\mathbf{x}_k,\mathbf{u}_k) + g(\mathbf{x}_T)\mid \mathbf{x}_t = \mathbf{x}\Big],
\end{equation}
for which we propose a quadratic ansatz
\begin{equation}
    \label{eqn:ansatz}
    V_t(\mathbf{x}) = \frac12\mathbf{x}^T P_t \mathbf{x} + \mathbf{s}_t^T \mathbf{x} + c_t \quad
    \textnormal{ for all } t\in[T],
\end{equation}
where 
$\{P_t\}_{t \in [T]}$, $\{\mathbf{s}_t\}_{t \in [T]}$, and $\{c_t\}_{t \in [T]}$ 
are deterministic sequences of matrices, vectors, and scalars, respectively.
By the Bellman optimality equation, the value function \(V\) satisfies 
\begin{align}\label{eq:value_function}
    V_t(\mathbf{x}) & = \inf_{\mathbf{u}}\,\mathbb{E}\left[h_t(\mathbf{x}, \mathbf{u}) + V_{t+1}(A\mathbf{x} + B\mathbf{u}+ \mathbf{w}_t)\right] \\
    & = \inf_{\mathbf{u}}\Big[\frac12\mathbf{x}^TM_t\mathbf{x} + \frac12\mathbf{u}^TH_t \mathbf{u} + \mathbf{x}^TG_t\mathbf{u}+ (\mathbf{q}_t^T+\mathbf{s}_{t+1}^TA)\mathbf{x} \\
    & \quad + (\mathbf{r}_t^T+ \mathbf{s}_{t+1}^TB)\mathbf{u} + d_t + c_{t+1} +\frac{1}{2}\mathbb{E}(\mathbf{w}_t^TP_{t+1}\mathbf{w}_t)\Big],
\end{align}
where~$M_{t} := Q_t+A^TP_{t+1}A$, $H_{t} := R+B^TP_{t+1}B$, and $ G_{t} := N_t+A^TP_{t+1}B$. Solving for the infimum yields the optimal Markovian deceptive control 
\begin{equation}\label{eq:optimal_control}
    \mathbf{u}_t^* = -H_{t}^{-1}(G_{t}^T\mathbf{x}_t + \mathbf{r}_t + B^T\mathbf{s}_{t+1}),
\end{equation}
whose existence relies on the invertibility of \(H_t\), which is guaranteed in Theorem~\ref{thm:well_posedness} under an appropriate range of \(\lambda\).
Importantly, \(\mathbf{u}_t^*\) is linear in the state process~$\xv_t$, which justifies the form of the controllers 
in the hypotheses in~\eqref{eq:h0def} and~\eqref{eq:h1def}.

Substituting~\eqref{eq:optimal_control} into~\eqref{eq:value_function} yields the backward recursions
\begin{align}\label{eq:backward_recursion}
    P_{t} &= M_{t} - G_{t}H_{t}^{-1}G_{t}^T, \\
    \mathbf{s}_t &= \mathbf{q}_t + A^T\mathbf{s}_{t+1} - G_{t}H_{t}^{-1}(\mathbf{r}_t + B^T\mathbf{s}_{t+1}), \\
    c_{t} &= d_t+c_{t+1} - \frac12(\mathbf{r}_t + B^T\mathbf{s}_{t+1})^TH_{t}^{-1}(\mathbf{r}_t + B^T\mathbf{s}_{t+1})\\& \qquad\qquad\qquad\qquad\qquad+\frac{1}{2}\mathrm{Tr}\Big(P_{t+1}\begin{bmatrix}\Sigma_z^2 & \mathbf{0}_{n\times n}\\\mathbf{0}_{n\times n} & \Sigma_y^2\end{bmatrix}\Big),
\end{align}
with given terminal conditions
\begin{equation}\label{eq:terminal_conditions}
    P_T = \begin{bmatrix}
        T_v & \mathbf{0}_{n\times n} \\\mathbf{0}_{n\times n} & \mathbf{0}_{n\times n}
    \end{bmatrix}, \ \mathbf{s}_T = \begin{bmatrix}
        -T_v^T\bar{\mathbf{v}}_T \\ \mathbf{0}_{n\times 1}
    \end{bmatrix}, \  c_T = \frac{1}{2}\bar{\mathbf{v}}_T^TT_v\bar{\mathbf{v}}_T,
\end{equation}
which provide a semi-explicit characterization of the blue team's control design.

Theorem~\ref{thm:well_posedness} analyzes the solvability of the backward recursions in~\eqref{eq:backward_recursion} and provides a sufficient condition for the existence of the optimal deceptive control in~\eqref{eq:optimal_control}.

\begin{theorem}[Global well-posedness]\label{thm:well_posedness}
    For any given~\(T\in \mathbb{N}\), when \(\lambda\) takes a small enough value such that~$I_n - \lambda (\Sigma_y^2)^{-\frac12}R_\beta^{-1}(\Sigma_y^2)^{-\frac12}\succeq 0$, \(H_t\) is invertible for any \(t\in[T]\), and there exists a unique solution to~\eqref{eq:backward_recursion}.
\end{theorem}

\begin{proof}
See Appendix~\ref{apdx:pf_well_posedness}.
\end{proof}

Notably, the well-posedness condition reduces to $\lambda\in[0,R_\beta\Sigma_y^2]$ in the one-dimensional case (i.e., $n=1$). 
Based on this well-posedness result, Corollary~\ref{cor:two_special_cases} examines two special cases of the resulting one-dimensional deception model, when \(\lambda\) takes values at the boundary of the admissible range.

\begin{corollary}\label{cor:two_special_cases}
    For the one-dimensional (i.e., \(n=1\)) deception model given by~\eqref{eq:state_dynamics} and~\eqref{eq:blue_team_cost},
    when~\(\lambda =0\), the optimal deceptive control~\(\beta^* \equiv 0\);
    when~\(\lambda =R_\beta\Sigma_y^2\), the optimal deceptive control~\(\beta^*\) exactly aligns with hypothesis~\(\mathsf{H}_1\), i.e.,~\(\beta^*_t = F^b_t\mathbf{v}_t + F^c_t\mathbf{p}_t + \mathbf{f}^d_t\).    
    In addition, the first component of~$\uv^*_t$, namely~$\alpha^*_t$, is identical in both cases.
\end{corollary}

\begin{proof}
See Appendix~\ref{apdx:pf_two_special_cases}.
\end{proof}

As implied by Corollary~\ref{cor:two_special_cases}, our deception model provides a continuous interpolation between full commitment to the primary task and full commitment to deception, with \(\lambda\) interpreted as the interpolation parameter.
In particular, when the deception intensity~\(\lambda =R_\beta\Sigma_y^2\) attains its maximum value within the admissible range, the primary and deception tasks become effectively decoupled.
A multi-dimensional extension of Corollary~\ref{cor:two_special_cases} can be established in a similar way.

To quantitatively measure the amount of deception introduced, we define the deception measure
\begin{equation}\label{eq:misdirection_measurement}
    \DeceptionMeasure(\lambda) := \frac{d(\boldsymbol\beta^{\mathsf{H}_0}, \boldsymbol\beta^*(\lambda))}{d(\boldsymbol\beta^{\mathsf{H}_0}, \boldsymbol\beta^{\mathsf{H}_1})},
\end{equation}
where $d:\R^{n T}\times\R^{n T}\ni (\mathbf{a},\mathbf{b}) \mapsto \sqrt{\sum_{t=0}^{T-1} \norm{\mathbf{a}_t-\mathbf{b}_t}_2^2}\in \R_+$,~${\beta}^{\mathsf{H}_0}$ (resp. ${\beta}^{\mathsf{H}_1}$) denotes the optimal controller~$\beta^*$ under~$\mathsf{H}_0$ (resp.~$\mathsf{H}_1$), and~$\beta^*(\lambda)$ denotes the optimal controller $\beta^*$ induced by  a given deception intensity~$\lambda$. 
Intuitively, $\DeceptionMeasure$ measures the extent to which~$\beta^*(\lambda)$ shifts away from~$\beta^{\mathsf{H}_0}$ towards~$\beta^{\mathsf{H}_1}$ for a given value of~$\lambda$, thereby quantifying the amount of deception introduced. 
Using such notations, Corollary~\ref{cor:two_special_cases} implies~$\DeceptionMeasure(0) = 0$ and~$\DeceptionMeasure(R_\beta\Sigma_y^2) = 1$ when $n=1$.

\section{The Stealthy Deception Problem}\label{sec:stealthy_misdirection}

This section addresses Problem~\ref{prob:stealthy_misdirection}.
As noted in the Introduction, we consider a red team
that knows the blue team's deception patterns 
$\{F^b_t,F^c_t,\mathbf{f}^d_t\}_{t \in [T-1]}$.
This worst-case scenario for the blue team represents what a sophisticated adversary may know (e.g.,
through an intelligence source) and enables the analysis of deception under worst-case conditions.
The red team uses SPRT to test if the blue team is introducing deceptive perturbations based on such knowledge.
Specifically, given publicly known error probabilities~$a,b$, the red team computes the thresholds \(U\) and \(L\) according to Lemma~\ref{lem:sprt_parameters}, and implements the SPRT decision rule specified in Definition~\ref{def:sprt}.
Anticipating the red team's inference routine, the blue team adjusts its deception intensity~\(\lambda\), subject to the well-posedness condition (cf. Theorem~\ref{thm:well_posedness}),
to try to keep its 
deception undetected. 

To develop the blue team's strategy for selecting~$\lambda$, we analyze the SPRT statistic~$\log L^*_t:= \log L_t(\mathbf{x}^*)$ under the blue team's optimal deceptive control~$\mathbf{u}^*$ (cf.~\eqref{eq:optimal_control}), where~$\xv^*$ denotes the state process induced by~$\uv^*$.
Given the red team's SPRT decision rule (cf. Definition~\ref{def:sprt}), the blue team selects~\(\lambda\) such that at any time~\(t\), the probability that the red team rejects~\(\mathsf{H}_0\) is at most~\(\EPS\), i.e.,
\begin{equation}
\label{eq:stealthiness_condition}
\mathbb{P}(\log L^*_t \geq \log \upperthreshold) \leq \epsilon \quad \textnormal{ for all } t \in [T], 
\end{equation}
where~\(0<\EPS<<1\) is a prescribed detection tolerance level.
In other words, the red team is likely unable to draw a statistically significant conclusion that deceptive perturbations are present in the blue team's trajectories.
Clearly, a larger value of~\(\EPS\) yields a larger admissible set of~\(\lambda\), but may lead to less stealthy deception.

\subsection{Sampling-Based Approximation}\label{subsec:sampling_based_approximation}

When \(\EPS\) does not take a very small value, we numerically compute \(\lambda\) with
a grid search over its admissible range (cf. Theorem~\ref{thm:well_posedness}).
For a given \(\lambda\), the probability \(\mathbb{P}(\log L^*_t\geq \log \upperthreshold)\) can be approximated via Monte Carlo (MC) simulation~\cite{harrison2010introduction}, and the satisfaction of the constraint in~\eqref{eq:stealthiness_condition} can be checked using these approximations.
We refer to this methodology as the ``sampling-based approach'' for stealthy deception. The following lemma provides confidence intervals for our sampling-based approximations.

\begin{lemma}[Score Confidence Interval \cite{agresti1998approximate}]\label{lem:confidence_interval}
    For a Monte Carlo estimator $\hat{p}$ of \(p\) based on \(n\) independent samples, the \((1-c)\)-confidence interval is given by
    \begin{equation}\label{eq:score_confidence_interval}
        \Big(\frac{\hat{p} + \frac12\gamma - \theta}{1+\gamma}, \frac{\hat{p} + \frac12\gamma + \theta}{1+\gamma} \Big),
    \end{equation}      
    where the significance level \(c\in(0,1)\), 
    $z_{c/2}:= \mathcal{Q}^{-1}(\frac{c}{2})$, $\gamma:=z^2_{c/2}/n$, and~$\theta:=z_{c/2} \sqrt{({\hat{p}(1-\hat{p})+ z_{c/2}^2/4n})/{n}}$.   
\end{lemma}

Notably, this confidence interval provides an upper bound 
for~$\mathbb{P}(\log L^*_t\geq \log \upperthreshold)$
with high confidence, even when the event~\(\{\log L^*_t\geq \log \upperthreshold\}\) is not observed in the samples.
As a result, the sampling-based approach produces a numerical range of values of~\(\lambda\) under which the blue team's deception remains stealthy (i.e., \eqref{eq:stealthiness_condition} holds) with high confidence.

While the sampling-based approach provides a practical estimate of~\(\mathbb{P}(\log L^*_t\geq \log \upperthreshold)\), it is subject to numerical error and does not offer deterministic theoretical guarantees.
Moreover, accurate estimation becomes challenging when~\(\EPS\) is sufficiently small (e.g.,~$\epsilon\ll 10^{-3}$, see~\cite{au2001estimation}), as the event of interest becomes rare.
These limitations motivate the subsequent development of an analytical probabilistic bound.

\subsection{Probabilistic Bound}\label{sec:bound}

To facilitate the forthcoming
theoretical analysis, we introduce the following assumption without loss of generality.

\begin{assumption}
    \label{assu:sd_x0}
     In the hypothesis \(\mathsf{H}_1\) (cf.~\eqref{eq:h1def}), \(\mathbf{f}^d\equiv 0\).
\end{assumption}

Assumption~\ref{assu:sd_x0} implies that \(\mathbf{r} \equiv 0\) (cf.~\eqref{eq:matrices_for_new_cost}). 
Since $\mathbf{f}_t^d$ only contributes to a deterministic component in~\eqref{eq:log_L_T}, the subsequently developed probabilistic bound can be easily extended to the general case where $\mathbf{f}^d\not\equiv 0$. 

The analytical bound for~\(\mathbb{P}(\log L^*_t\geq \log \upperthreshold)\) established below is primarily based on concentration inequalities (of Hanson–Wright type \cite{vershynin2018high}), namely that~\(\log L^*_t\) does not deviate significantly from its mean~\(\E \log L^*_t\).
Consequently, the following technical result is required, which provides a tractable expression for~$\mathbb{E}\log L_t^*$.

\begin{lemma}\label{prop:E_log_L_t_star} 
    Under Assumption~\ref{assu:sd_x0}, $\xv_k^*\sim\mathcal{N}(m^*_k, \Sigma^*_k)$ for all~$k \in [T]$, where $m_k^*$ and $\Sigma_k^*$ satisfy the recursions 
    \begin{align}
        \label{eq:mean_x_k}
        &m^*_k = K_{k-1}^Tm^*_{k-1}- H^{-1}_{k-1}\mathbf{s}_k,\\
        &\Sigma^*_k = K_{k-1}^T(\Sigma^*_{k-1}+m^*_{k-1}(m^*_{k-1})^T)K_{k-1}- K_{k-1}^Tm^*_{k-1}\mathbf{s}_k^TH_{k-1}^{-1} \\
        \label{eq:cov_x_k}
        &\hspace{-0.75em}-H_{k-1}^{-1}\mathbf{s}_k(m^*_{k-1})^TK_{k-1}+ H_{k-1}^{-1}\mathbf{s}_k\mathbf{s}_k^TH_{k-1}^{-1}+ \Sigma_w- m^*_k(m^*_k)^T,
    \end{align}
    with given initial conditions~$m_0^* = \mathbf{x}_0$ and~$\Sigma_0^* = \mathbf{0}_{2n\times2n}$.
    Here,~$K_{k-1} := A^T - G_{k-1}H^{-1}_{k-1}$ and~$\Sigma_w := \mathrm{diag}(\Sigma_z^2, \Sigma_y^2)$.
    
    In addition,~$\E\log L_t^*$ can be represented as
    \begin{multline}\label{eqn:E_log_L_star}
        \mathbb{E}\log L_t^* = -\frac12\sum_{k=0}^{t-1}\Big\{2\mathbf{s}_{k+1}^TH_k^{-1}E_2(\Sigma_y^2)^{-1}F_k^T m_k^* \\
         +\mathrm{Tr}\Big[[F_k(\Sigma_y^2)^{-1}F_k^T + 2G_kH_k^{-1}E_2(\Sigma_y^2)^{-1}F_k^T]\\
         [\Sigma^*_k + m^*_k(m^*_k)^T]\Big]\Big\},
    \end{multline}
    where~$ E_2 := \begin{bmatrix}\mathbf{0}_{n\times n} & I_n\end{bmatrix}^T$ and~$ F_k := \begin{bmatrix}F^b_k & F^c_k\end{bmatrix}^T$.
\end{lemma}

\begin{proof}
See Appendix~\ref{apdx:pf_E_log_L_t_star}.
\end{proof}

To prepare for stating the probabilistic bound on $\mathbb{P}(\log L^*_t \geq \log \upperthreshold)$, we define a vector
\begin{multline}
        \mathcal{L}^1_t := \mathrm{Concat}\Big(\mathbb{I}_{\{k\geq 0\}} E_2(\Sigma_y^2)^{-1}F_k^T\mathcal{D}_{k-1}\\ - \mathbb{I}_{\{k\leq t-2\}}\sum_{j=k+1}^{t-1}\mathcal{C}_{k+1,j-1}^T\Big[F_j(\Sigma_y^2)^{-1}E_2^TH_j^{-1}\mathbf{s}_{j+1}\\  \qquad\quad + \mathcal{S}_j^s\mathcal{D}_{j-1}\Big],
        \ \forall -1\leq k\leq t-1\Big)\in \R^{2n(t+1)}, 
    \end{multline}
and a matrix~$\mathcal{L}_t^{2}\in \R^{2n(t+1)\times 2n(t+1)}$ consisting of \((t+1)^2\) blocks of $2n \times 2n$ matrices, with its $(i,j)^{th}$ block defined as
    \begin{multline}\label{eq:L_2_component}
        (\mathcal{L}^2_t)_{ij} := \frac{1}{2}\big(2\mathbb{I}_{\{0\leq i\leq t-1,-1\leq j\leq i-1\}}E_2(\Sigma_y^2)^{-1}F_i^T\mathcal{C}_{j+1,i-1} \\ - \mathbb{I}_{\{i\vee j\leq t-2\}}\sum_{k=i\vee j+ 1}^{t-1}\mathcal{C}_{i+1,k-1}^T\mathcal{S}_k^T\mathcal{C}_{j+1,k-1}\big)\in \R^{2n\times2n},
    \end{multline}
    for $i,j\in\{-1,0,\dots, t-1\}$.
    Here, the coefficients are defined in terms of the model parameters as
    \begin{align}
        &\mathcal{C}_{j+1,k-1} :=K_{k-1}^T\ldots K_{j+1}^T\in \mathbb{R}^{2n\times 2n}, \\ 
        &\mathcal{D}_{k-1} := -\sum_{j=0}^{k-1}\mathcal{C}_{j+1,k-1}H_j^{-1}\mathbf{s}_{j+1}\in \R^{2n},\\
        &\mathcal{S}_k := F_k(\Sigma_y^2)^{-1}F_k^T + 2G_kH_k^{-1}E_2(\Sigma_y^2)^{-1}F_k^T\in \mathbb{R}^{2n\times 2n}.
    \end{align}
    
    In Theorem~\ref{thm:probabilistic_bound}, the probabilistic bound is explicitly in terms of the above quantities.

\begin{theorem}\label{thm:probabilistic_bound}
    Under Assumption~\ref{assu:sd_x0},
    \begin{multline}
        \mathbb{P}(\log L^*_t \geq \log \upperthreshold)\leq \mathrm{exp}\Big\{-\frac{\EPS_1^2}{2\|(\Sigma^w_t)^{\frac12}(\mathcal{L}^1_t + 2\mathcal{L}^{2,s}_tm^w_t)\|_2^2}\Big\} \\
        + \mathrm{exp}\Big\{-\frac18\Big(\frac{\EPS_2^2}{\|(\Sigma^w_t)^{\frac12}\mathcal{L}^{2,s}_t(\Sigma^w_t)^{\frac12}\|_F^2} \wedge \frac{\EPS_2}{\|(\Sigma^w_t)^{\frac12}\mathcal{L}^{2,s}_t(\Sigma^w_t)^{\frac12}\|_2}\Big)\Big\},
    \end{multline}
    for any $\EPS_1, \EPS_2>0$ that satisfy $\EPS_1+\EPS_2 = \log \upperthreshold - \E\log L^*_t$,
    where \(\E\log L^*_t\) can be explicitly computed using the model parameters according to Lemma~\ref{prop:E_log_L_t_star}.
    Here, $\mathcal{L}_t^{2,s} := [\mathcal{L}_t^{2} + (\mathcal{L}_t^{2})^T]/2$ is the symmetric part of $\mathcal{L}_t^{2}$ and 
    \begin{align}
        &m^w_t :=  \mathrm{Concat}(m^*_0,0,\ldots,0)\in \R^{2n(t+1)},\\
        &\Sigma^w_t :=  \mathrm{diag}(\Sigma^*_0,\Sigma_z^2,\Sigma_y^2,\ldots,\Sigma_z^2,\Sigma_y^2)\in \mathbb{R}^{2n(t+1)\times 2n(t+1)}.
    \end{align}
\end{theorem}

\begin{proof}
See Appendix~\ref{apdx:pf_probabilistic_bound}.
\end{proof}

The bound in Theorem~\ref{thm:probabilistic_bound} represents a worst-case estimate of the probability that the red team detects the blue team's deception. 
Based on this result, the blue team can perform a grid search to find admissible values of~$\lambda$ under which the constraint~\eqref{eq:stealthiness_condition} is satisfied.
We remark that the resulting choice of \(\lambda\) may be conservative, as the concentration bound in Theorem~\ref{thm:probabilistic_bound} may not be tight in all cases.

%% file: 4-Results-E.tex
\section{Numerical Experiments}\label{sec:results}

This section addresses Problem~\ref{prob:empirical_verification}.
We first demonstrate the impact of deception on the blue team's state trajectory, 
and then we discuss the selection of the deception intensity $\lambda$ to avoid detection by the red team.
Then, we verify the stealthiness of deception under the chosen $\lambda$, and explore
how~$\lambda$ changes in terms of the detection tolerance~$\epsilon$.

For all experiments, we adopt the following set of model parameters unless otherwise specified:
\begin{equation} \label{eq:sim_parameters}
    \begin{aligned}
    &n = 1, \ T=20,\ \Sigma_z^2 = \Sigma_y^2 = 0.05,\ \mathbf{v}_0 = 1, \ \mathbf{p}_0 = 4, \\ 
     &\ R_\alpha = 1, \ R_\beta = 10,\ R_v = 1, \ T_v = 1, \ \bar{\mathbf{v}}_t = 1 - t/T, \\
    &\ \ F^b \equiv 0.5, \ F^c \equiv -0.1,\ \mathbf{f}^d \equiv 0, \ a=b=0.01.
    \end{aligned}
\end{equation}

\subsection{Performance of the Optimal Deceptive Control}\label{sec:result_misdirection}

In this section, we demonstrate the impact of the optimal deceptive control \(\mathbf{u}^*\) (cf. \eqref{eq:optimal_control}) on the blue team's trajectories. Figure~\ref{fig:trajectories_with_misdirection} compares the blue team's state and control trajectories under different values of \(\lambda\). 
As observed, in the baseline case (\(\lambda = 0\)), deception does not exist (i.e., \(\beta^*\equiv 0\)), which aligns with the theoretical insights from Corollary~\ref{cor:two_special_cases}. 
When $\lambda>0$, the state and control trajectories deviate from the baseline ones, indicating the presence of deception.
A larger value of~$\lambda$ leads to stronger deception, as reflected by the larger magnitude of \(\beta^*\).

Figure~\ref{fig:primary_cost_with_lambda} exhibits the trade-off between the primary cost 
and the deception measure~\eqref{eq:misdirection_measurement} under different values of $\lambda$. 
In what follows, the primary cost refers to the cost (cf.~\eqref{eq:primary_cost}) evaluated along a single simulated trajectory.
As~$\lambda$ increases, both the primary cost and the deception measure increase.
A higher deception intensity allocates more control effort to deception, thereby reducing the effort devoted to minimizing~$J^{\mathrm{primary}}$.

Figure~\ref{fig:logLt_with_lambda} plots the trajectories of $\log L_t^*$ computed by the red team.
In the baseline case, $\log L_t^*$ takes highly negative values, indicating the red team's acceptance of $\mathsf{H}_0$.
When~\(\lambda\) takes small values (e.g., 0.1),~$\log L_t^*$ never reaches~\(\log U\), implying that the blue team's deception remains undetected.
In contrast, when~\(\lambda\) further increases (e.g.,~\(\lambda\geq 0.2\)),~$\log L_t^*$ reaches the threshold~\(\log U\), enabling the red team to conclude that deceptive perturbations are present.

\begin{figure}[htbp]
    \centering
    \begin{subfigure}{0.48\textwidth}
        \includegraphics[width=\linewidth]{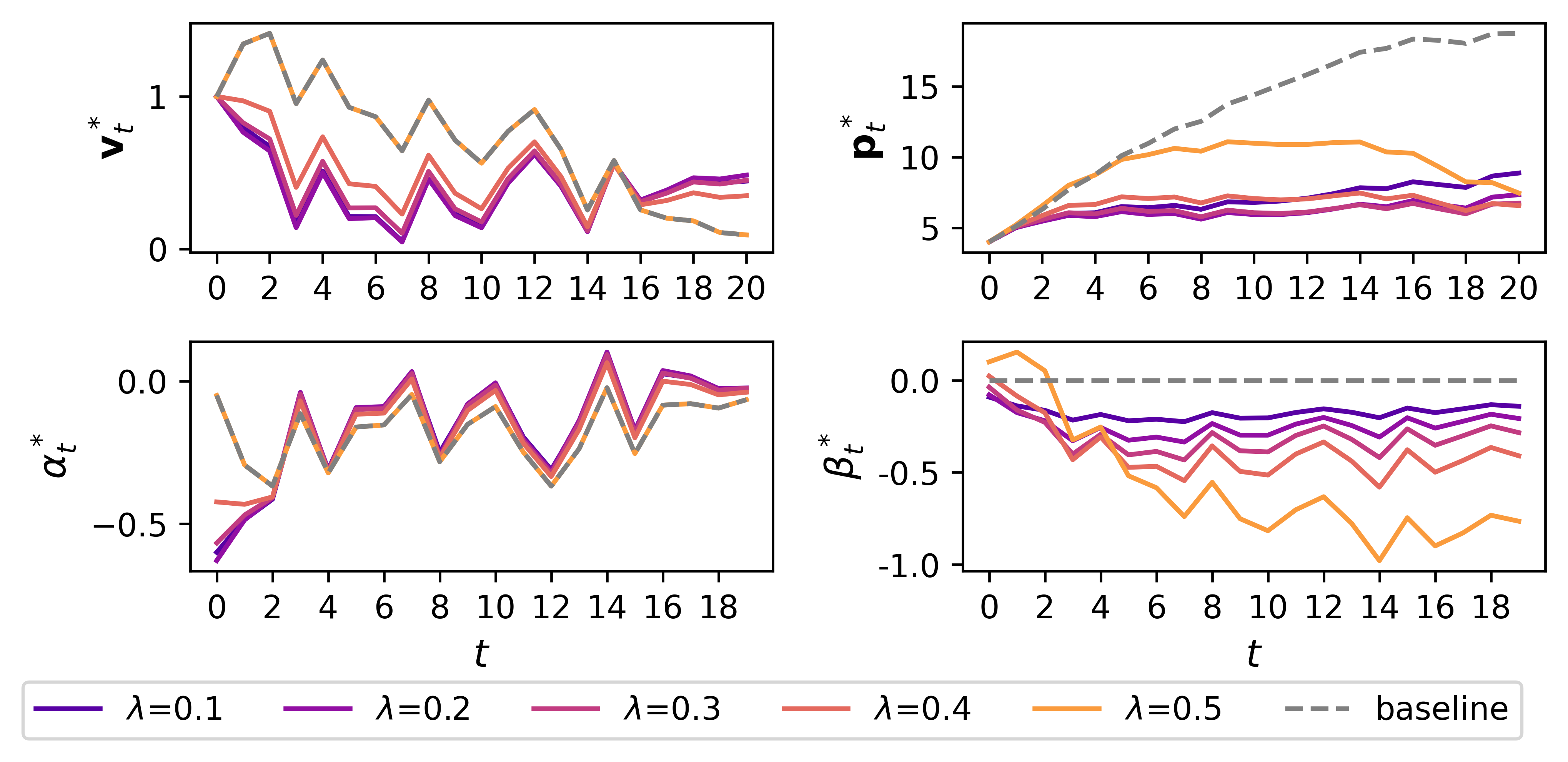}
        \caption{The blue team's optimal state and control trajectories.}
        \label{fig:trajectories_with_misdirection}
    \end{subfigure}
    \begin{subfigure}{0.23\textwidth}
        \includegraphics[width=\linewidth]{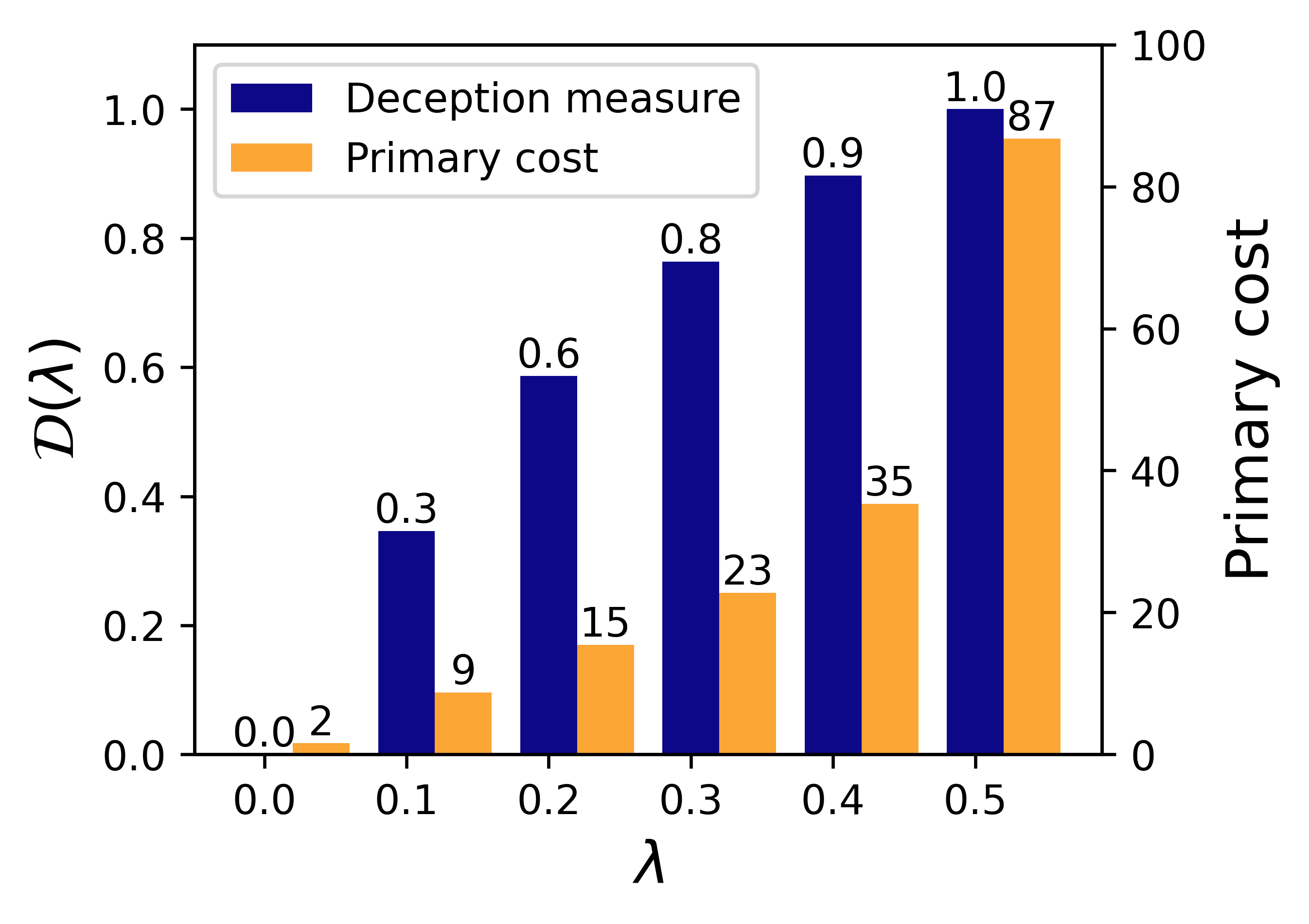}
        \caption{Values of the deception measure and the primary cost.}
        \label{fig:primary_cost_with_lambda}
    \end{subfigure}
    \begin{subfigure}{0.22\textwidth}
        \includegraphics[width=\linewidth]{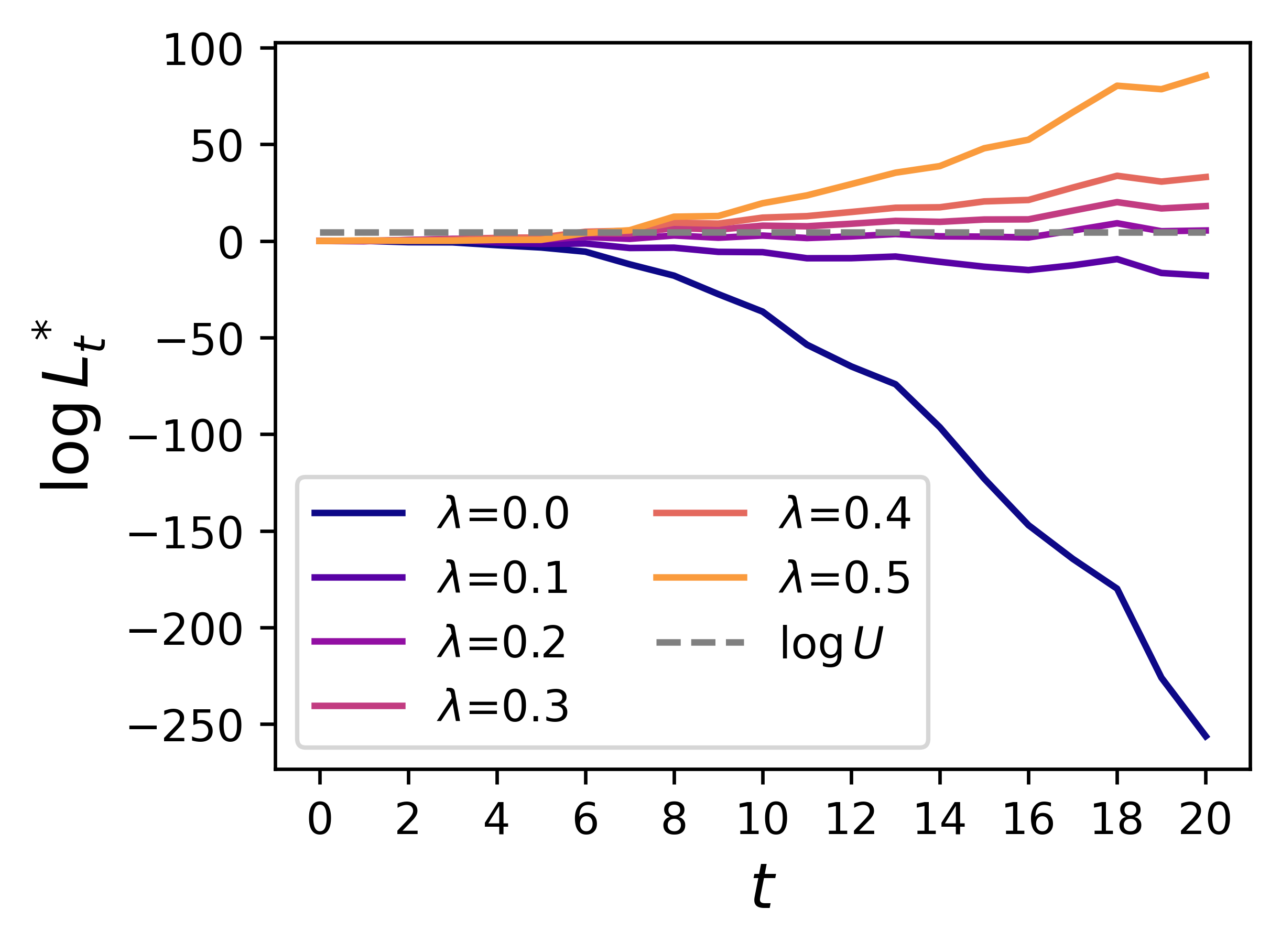}
        \caption{Trajectories of $\log L^*_t$ with detection threshold $\log \upperthreshold$.}
        \label{fig:logLt_with_lambda}
    \end{subfigure}
    \caption{
    Comparison of trajectories, primary costs, and deception measures under different values of~$\lambda$.
    Trajectories are generated by identical realizations of \(\{\mathbf{w}_t\}_{t\in[T-1]}\).}
    \label{fig:misdirection_with_lambda}
\end{figure}

\subsection{Intensity Selection for Stealthy Deception} \label{sec:result_stealthiness}

In this section, we demonstrate the blue team's selection strategy for the intensity~$\lambda$ that induces stealthy deception. 
To find~$\lambda$ under which the stealthiness condition~\eqref{eq:stealthiness_condition} holds, 
we perform a grid search over all admissible~\(\lambda\) satisfying the well-posedness condition in Theorem~\ref{thm:well_posedness}.

\begin{figure}[htbp]
    \centering
    \includegraphics[width=1\linewidth]{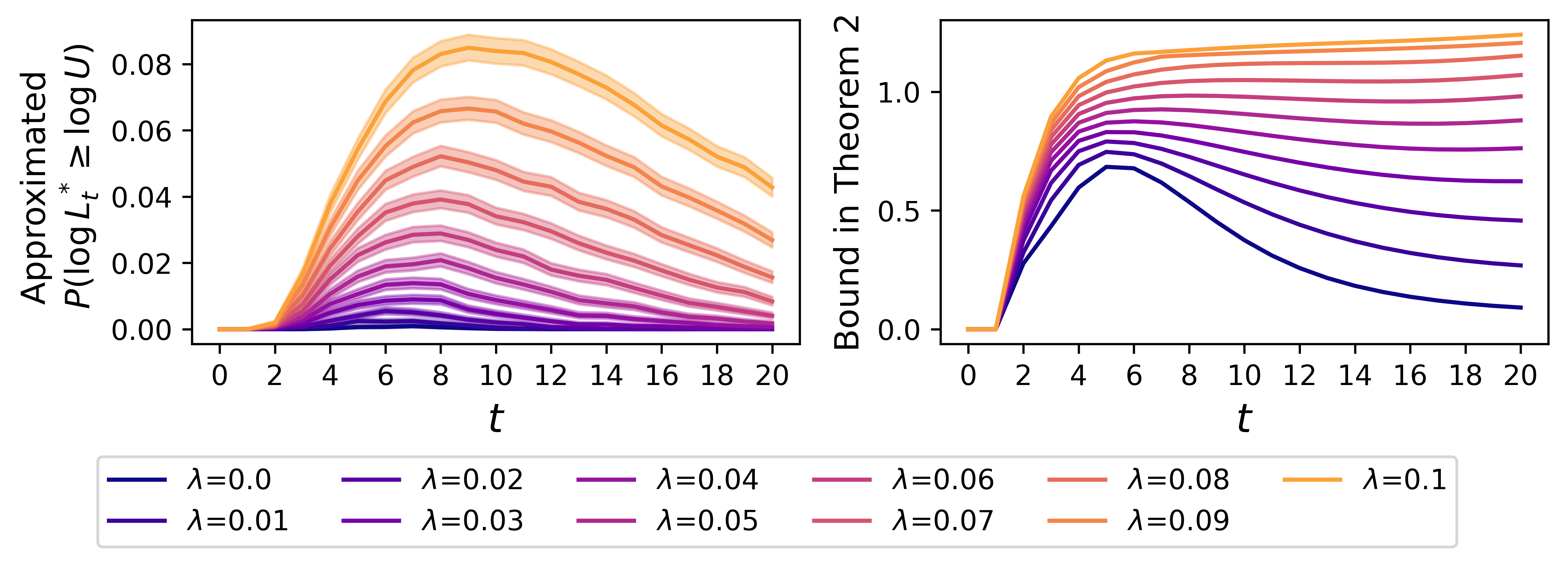}
    \caption{Approximations,~\(95\%\)-confidence intervals, and analytical bounds for~$\mathbb{P}(\log L^*_t\geq\log \upperthreshold)$ in the stealthiness constraint~\eqref{eq:stealthiness_condition}. 
    In the left panel, solid lines represent sampling-based approximations (computed based on~\(20,000\) samples), and colored areas represent~$95\%$-confidence intervals~\eqref{eq:score_confidence_interval}.}
    \label{fig:compare_logLt_and_proabability}
\end{figure}

\begin{figure}[htbp]
    \centering
    \begin{subfigure}{0.48\textwidth}
        \includegraphics[width=\linewidth]{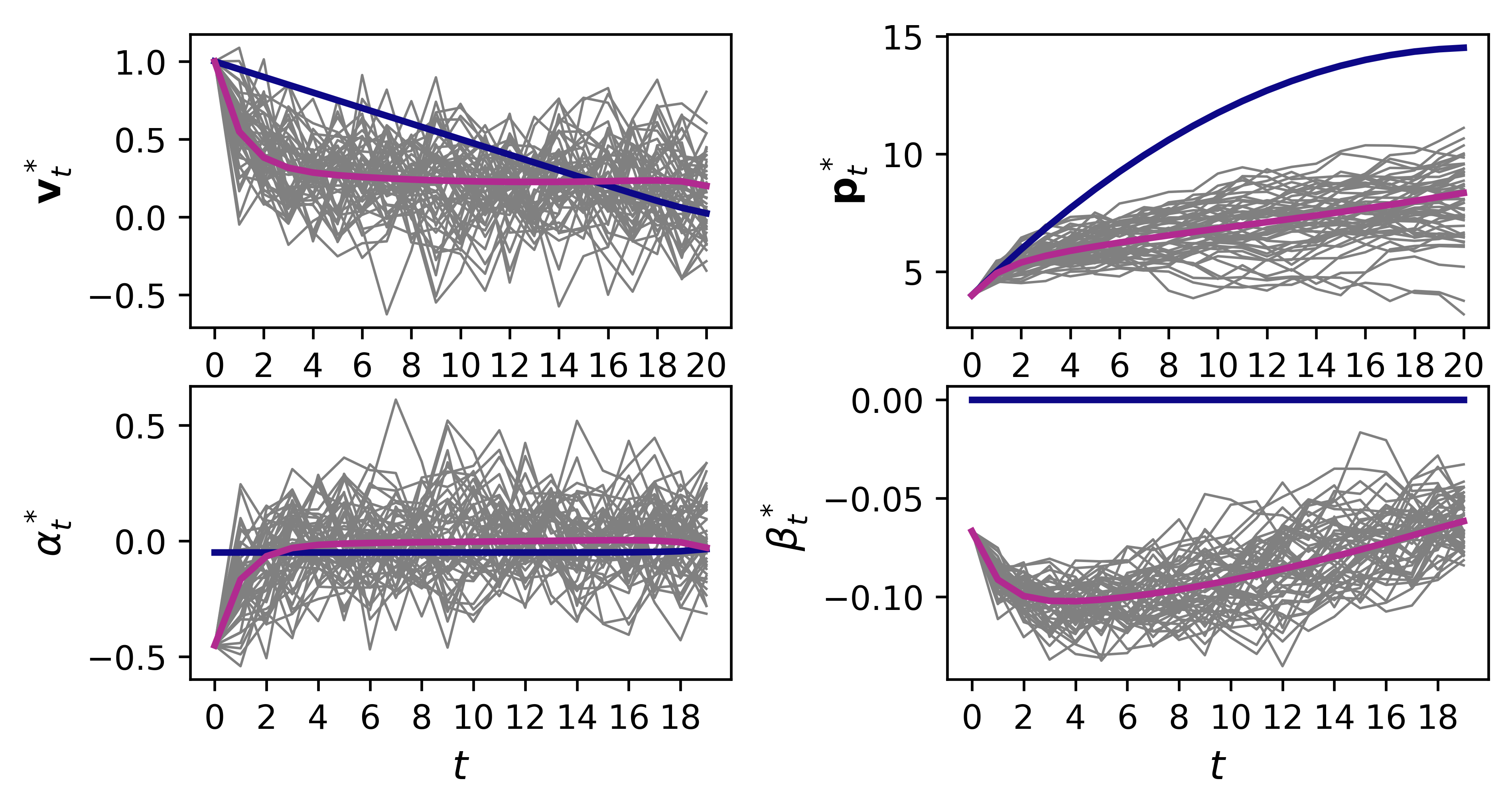}
        \caption{Optimal state and control trajectories.
        }
        \label{fig:random_trajectories}
    \end{subfigure}
    \begin{subfigure}{0.3\textwidth}
        \includegraphics[width=0.9\linewidth, height = 3cm]{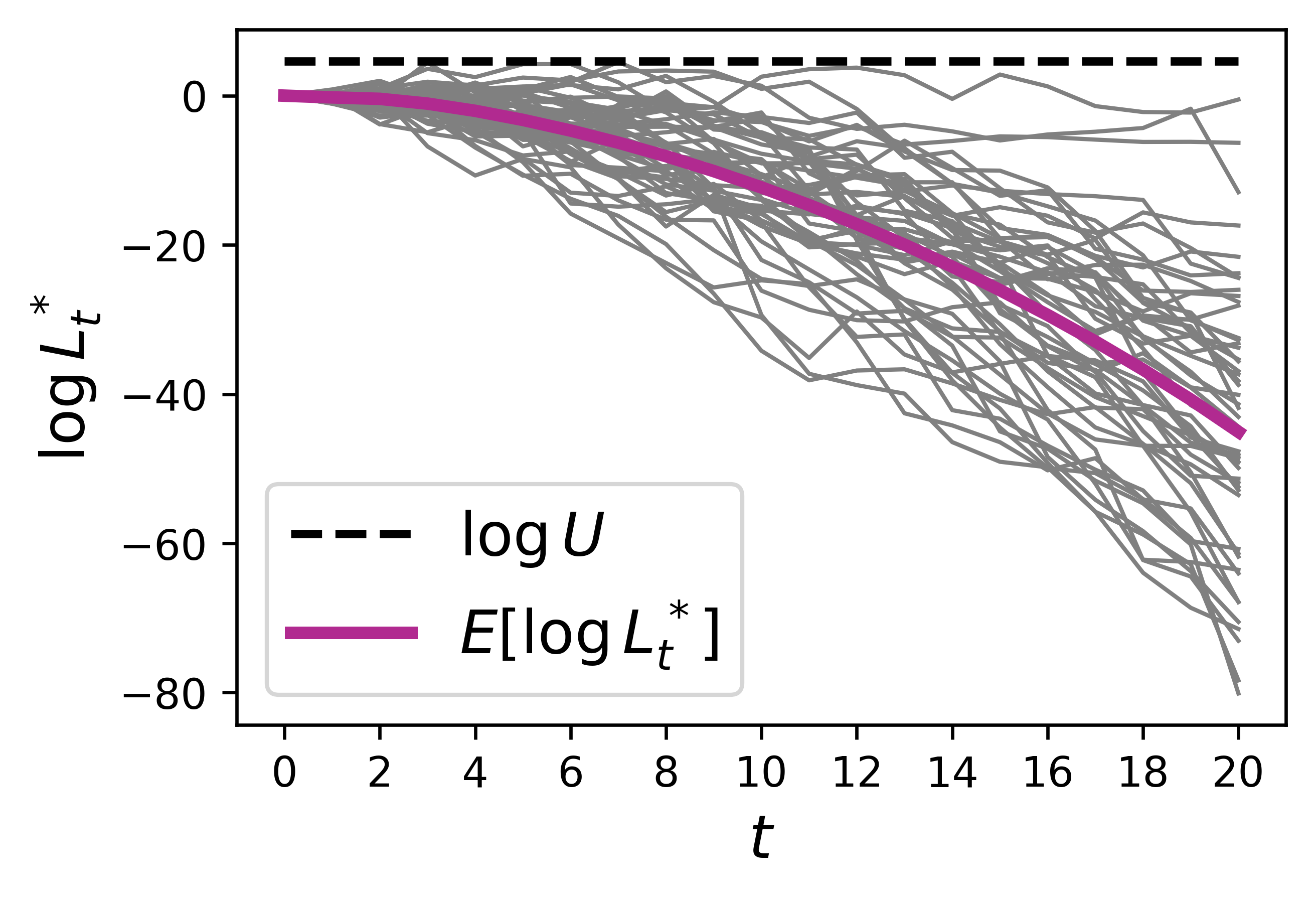}
        \caption{Trajectories of \(\log L^*_t\) under~$\lambda = 0.04$.}
        \label{fig:random_logLt}
    \end{subfigure}
    \caption{Comparison of trajectories generated by $\lambda = 0$ and $\lambda = 0.04$.
    Blue lines represent mean baseline trajectories; red  lines represent mean trajectories under \(\lambda = 0.04\); gray lines represent \(500\) independent trajectories under \(\lambda = 0.04\).}
    \label{fig:random_results_with_fixed_lambda}
\end{figure}

The left panel of Figure~\ref{fig:compare_logLt_and_proabability} presents sampling-based approximations of $\mathbb{P}(\log L^*_t\geq\log \upperthreshold)$ and the associated $95\%$-confidence intervals (cf. Lemma~\ref{lem:confidence_interval}) under different $\lambda\in[0,0.1]$.
As observed, the blue team can select the deception intensity $\lambda = 0.04$, under which the \(95\%\)-confidence interval remains below the detection tolerance $\epsilon = 0.02$.
In other words, the blue team's deceptive perturbations induced by the intensity $\lambda = 0.04$ have
less than a \(2\%\) chance of being detected by the red team at each time step \(t\in[T]\).

The right panel of Figure~\ref{fig:compare_logLt_and_proabability} plots the analytical probabilistic bound established in Theorem~\ref{thm:probabilistic_bound} under the same set of model parameters.
Clearly, the analytical bound is conservative.
For example, using Theorem~\ref{thm:probabilistic_bound}
the blue team's deceptive perturbations induced by the same intensity $\lambda = 0.04$ 
are bounded as having less than a \(90\%\) chance of being detected by the red team.
This lack of tightness mainly comes from using the union bound in the proof of Theorem~\ref{thm:probabilistic_bound}, which neglects the correlation between the linear and quadratic forms.
This approach is adopted to maintain analytical tractability, as capturing such dependencies would lead to significantly more involved technical arguments.
However, this analytical bound is useful when (i) a theoretical verification of the stealthiness constraint is required and/or (ii) the sampling-based approach fails for a sufficiently small~$\epsilon$.

Figure~\ref{fig:random_results_with_fixed_lambda} verifies the effectiveness of the selected intensity~$\lambda = 0.04$, suggested by the sampling-based approach, by plotting the resulting trajectories. 
As observed, the optimal state and control trajectories deviate from their baseline counterparts, while for almost all of the~\(500\) independently generated sample paths,~$\log L^*_t$ never reaches the threshold~$\log \upperthreshold$.
This implies that the blue team has introduced deceptive perturbations into its trajectories, but these perturbations remain stealthy with high probability.

\begin{figure}[htbp]
    \centering
    \includegraphics[width=0.7\linewidth, height = 4cm]{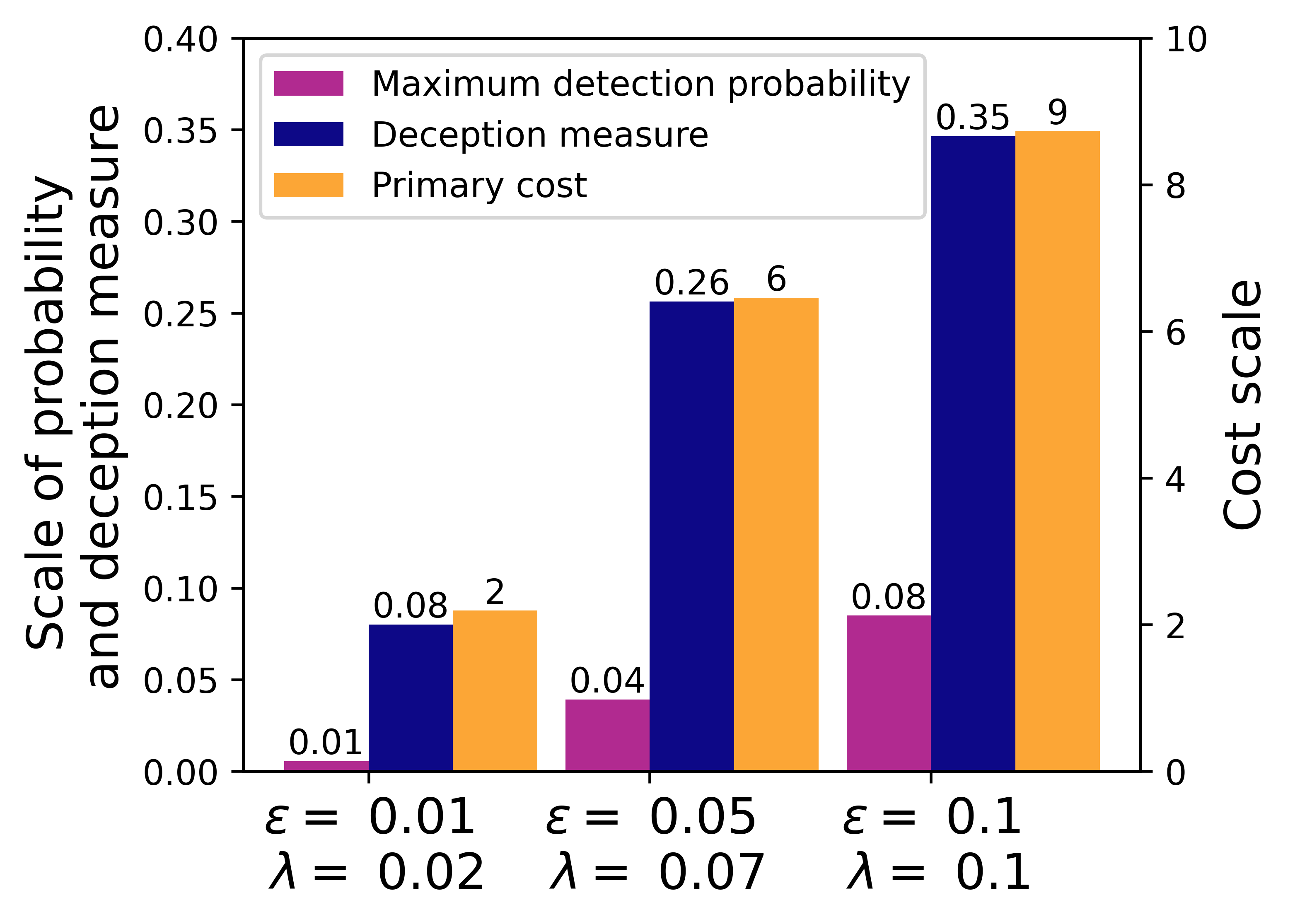}
    \caption{Trade-off among detection probability, deception measure, and primary cost under different detection 
    tolerances \(\EPS\).
    Here, $\lambda$ is chosen as the largest value (identified via the sampling-based approach) that satisfies the stealthiness constraint~\eqref{eq:stealthiness_condition}.
    }
    \label{fig:compare_different_epsilon}
\end{figure}

Lastly, Figure~\ref{fig:compare_different_epsilon} illustrates the trade-off among three important quantities: 
(i) the maximum detection probability~$\max_{t\in[T]}\mathbb{P}(\log L^*_t\geq\log \upperthreshold)$, 
(ii) the deception measure~$\DeceptionMeasure(\lambda)$, and 
(iii) the primary cost, where $\lambda$ is chosen as the largest value satisfying the stealthiness constraint~\eqref{eq:stealthiness_condition} with $95\%$ confidence under different tolerance levels \(\EPS\). 
We observe that, for larger $\epsilon$, the corresponding $\lambda$ is larger, allowing for stronger deception at the price of a higher primary cost and increased detectability.

%% file: 5-Conclusion.tex
\section{Conclusion}\label{sec:conclusion}

In this work, we developed strategies for stealthy deception in linear-quadratic control. 
We augmented the standard LQ cost with a likelihood ratio term
and derived both the optimal control law with dynamic programming. 
For the stealthy deception problem, both a sampling-based approach and a probabilistic bound were developed 
to guide selection of the deception intensity.
Numerical results show that the proposed framework successfully deceives 
with only a limited increase in the primary cost, while maintaining stealthiness with high probability.
Future work will investigate how deception affects the red team’s estimates of the blue team’s data (e.g., its cost), and will also study game settings in which the red team can react to and counter the blue team’s deception.

%% file: 6-Appendix-E.tex
\appendix

\subsection{Supporting Lemmas}\label{apdx:supporting_lemmas}

The following lemmas will be used in Appendix~\ref{apdx:pf_probabilistic_bound} to prove Theorem~\ref{thm:probabilistic_bound}.

\begin{lemma} \label{lem:HW}
    Let \(X\sim \mathcal{N}(\mathbf{0},\Sigma_X)\) be a Gaussian random vector with positive definite covariance matrix \(\Sigma_X\in\mathbb{R}^{n\times n}\).    
    For a given matrix \(A\in\mathbb{R}^{n\times n}\), 
    \begin{multline}
    \mathbb{P}(X^TAX - \E(X^TAX)\geq t) \leq \\
        \exp\left\{{-\frac18\left(\frac{t^2}{\|\Sigma_X^{\frac12}A\Sigma_X^{\frac12}\|_F^2} \wedge \frac{t}{\|\Sigma_X^{\frac12}A\Sigma_X^{\frac12}\|_2}\right)}\right\}
    \end{multline}
    for any $t>0$, where \(\|\cdot\|_2\) denotes the matrix 2-norm and \(\|\cdot\|_F\) denotes the Frobenius norm.
\end{lemma}

\begin{proof}
    The proof follows the strategy used to establish the general Hanson–Wright inequality in \cite[Section~6]{vershynin2018high}.
    
    Without loss of generality, we prove the inequality for the special case \(\Sigma_X=I_n\), i.e.,
    \begin{equation}
        \mathbb{P}(X^TAX - \E(X^TAX)\geq t)\leq e^{-\Big(\frac{t^2}{8\|A\|_F^2}\wedge \frac{t}{8\|A\|_2}\Big)},\ \forall t > 0,
    \end{equation}
    where~\(X\sim \mathcal{N}(0,I_n)\).
    The general result can be easily proved by applying the inequality above to \(\Sigma_X^{-\frac12}X\sim \mathcal{N}(0,I_n)\) while setting \(A\) as \(\Sigma^{\frac12}A\Sigma^{\frac12}\).
    
    Firstly, note that~\(X^TAX = X^TA^sX\), where the symmetric part \(A^s\) admits the eigenvalue decomposition \(A^s = W^T\Sigma W\), where~\(W\) is unitary and~\(\Sigma\) is diagonal.
    As a result, \(X^TAX = Y^T\Sigma Y\), where~\(Y := WX\sim \mathcal{N}(0,I_n)\).
    
    Since \(\E X^TAX = \E \sum_{i=1}^n \Sigma_{ii} Y_i^2 = \mathrm{Tr}(\Sigma)\), define~\(Z := X^TAX - \E X^TAX = \sum_{i=1}^n \Sigma_{ii} (Y_i^2 - 1)\). Then, the Chernoff bound~\cite[Lemma~2.3]{zhang2020concentration}
    provides~$\mathbb{P}(Z\geq t)\leq e^{-\phi t}\E e^{\phi Z}$, where~$\E e^{\phi Z} = \prod_{i=1}^ne^{-\phi \Sigma_{ii}}\E e^{\phi  \Sigma_{ii} Y_i^2}$.
    Since~$\E e^{tY_i^2} = \frac{1}{\sqrt{1-2t}}$ 
    for~\(Y_i\sim \mathcal{N}(0,1)\)
    and any $t < \frac12$, $\E e^{\phi\Sigma_{ii} Y_i^2} = ({1-2\phi\Sigma_{ii}})^{-\frac12}$ for any 
    $|\phi| <\frac{1}{2|\Sigma_{ii}|}$.
    Consequently,~\(\E e^{\phi Z} = e^{-\phi \mathrm{Tr}(\Sigma)}\prod_{i=1}^n(1-2\phi\Sigma_{ii})^{-\frac12}\) and
    \begin{multline}
        \log \mathbb{P}(Z\geq t)\leq  -\phi t -  \frac12\sum_{i=1}^n [2\phi \Sigma_{ii} + \log(1-2\phi\Sigma_{ii})] \\ \leq -\phi t + 2\phi^2\sum_{i=1}^n (\Sigma_{ii})^2 
        \textnormal{ for all } |\phi|\leq \frac{1}{4\max_i|\Sigma_{ii}|},
    \end{multline}
    where the inequality follows from $-\log(1-u) - u\leq u^2,\ \forall|u|\leq \frac12$.
    Since~\(\max_i|\Sigma_{ii}| = \|\Sigma\|_2 = \|A^s\|_2\leq \|A\|_2\) and~\(\sum_{i=1}^n (\Sigma_{ii})^2 = \|\Sigma\|_F^2 = \|A^s\|_F^2\leq \|A\|_F^2\), the upper bound is minimized at
    \begin{equation}
        \phi = \frac{t}{4\|A\|_F^2}\wedge \frac{1}{4\|A\|_2}.
    \end{equation}
    Plugging in this value of \(\phi\) yields
    \begin{equation}
        \mathbb{P}(Z\geq t)\leq e^{-(\frac{t^2}{8\|A\|_F^2}\wedge \frac{t}{8\|A\|_2})},
    \end{equation}
    which concludes the proof.
\end{proof}

\begin{lemma}[Gaussian Concentration~\cite{boucheron2013concentration}]
    \label{lem:Gaussian_Concen}
    Let \(X\sim \mathcal{N}(0,I_n)\) be a Gaussian random vector, and let \(F:\R^n\to\R\) be an \(L\)-Lipschitz function under the Euclidean norm~$\|\cdot\|_2$. Then
    \begin{align}
        \mathbb{P}(F(X) - \E [F(X)]\geq \EPS)\leq e^{-\frac{\EPS^2}{2L^2}},\ \forall \EPS>0.
    \end{align}
\end{lemma}

\subsection{Proof of Theorem~\ref{thm:well_posedness}}\label{apdx:pf_well_posedness}

    Firstly, we prove that, if $I_n - \lambda (\Sigma_y^2)^{-\frac12}R_\beta^{-1}(\Sigma_y^2)^{-\frac12}\succeq 0$, then the quadratic terms in~\eqref{eq:modified_running_cost} remain positive, i.e.,
    \begin{equation}
        \label{eqn:claim1}
        \frac12\mathbf{x}^TQ_t\mathbf{x} + \frac12\mathbf{u}^T R \mathbf{u} + \mathbf{x}^T N_t \mathbf{u}\geq  0,\ \forall(\mathbf{x},\mathbf{u}).
    \end{equation}
    Recall that $R_\alpha, R_\beta, R_v$ are positive definite and
    \begin{equation}
    \begin{aligned}
        Q_t-N_tR^{-1}N_t^T& = \begin{bmatrix}
            R_v+\lambda (F_t^b)^TW(\lambda)F_t^b & \lambda (F_t^b)^TW(\lambda)F_t^c \\
            \lambda (F_t^c)^TW(\lambda)F_t^b & \lambda (F_t^c)^TW(\lambda)F_t^c
        \end{bmatrix}\\
        & = \begin{bmatrix}
            R_v & \mathbf{0}_{n\times n} \\ \mathbf{0}_{n\times n}& \mathbf{0}_{n\times n}
        \end{bmatrix} + \bar{F}_t^TW(\lambda)\bar{F}_t,
    \end{aligned}
    \end{equation}
    where $\bar{F}_t := \begin{bmatrix}
        F_t^b & F_t^c
    \end{bmatrix}\in\R^{n\times 2n}$ and $W(\lambda):=\lambda(\Sigma_y^2)^{-\frac12}[I_n - \lambda (\Sigma_y^2)^{-\frac12}R_\beta^{-1}(\Sigma_y^2)^{-\frac12}](\Sigma_y^2)^{-\frac12}\succeq 0$.
    As a result, $Q_t-N_tR^{-1}N_t^T\succeq 0$.
    By the Schur complement~\cite[Appendix 5.5]{boyd2004convex}, since \(R\succ 0\), \(\begin{bmatrix}
            Q_t & N_t \\ N_t^T & R
        \end{bmatrix} \succeq 0,\ \forall t \in [T-1]\), the claim~\eqref{eqn:claim1} is proved.

    Then, we show that~\eqref{eqn:claim1} implies $P_t\succeq0$ for any $t\in[T]$.
    We prove this claim by contradiction, assuming that there exists~$t_0 \in [T]$ such that~$P_{t_0}$ is not positive semi-definite.
    In other words, there exist~$t_0 \in [T]$ and~$\mathbf{y}_0 \neq 0$ such that~$\mathbf{y}_0^TP_{t_0}\mathbf{y}_0<0$.
    Since~$\frac{\mathbf{s}_{t_0}^Ta\yv_0+c_{t_0}}{(a\yv_0)^TP_{t_0}(a\yv_0)} \rightarrow 0$ 
    as~$a \rightarrow \infty$,     
    there exists~$a_0 > 0$     
    such that~$\mathbf{s}_{t_0}^T(a\yv_0)+c_{t_0} \leq \frac14|(a\yv_0)^TP_{t_0}(a\yv_0)|$ for any $a>a_0$.    
    Combining with the ansatz~\eqref{eqn:ansatz}, we yield
    \begin{multline}
        V_{t_0}(a\yv_0) = \frac12(a\yv_0)^TP_{t_0}(a\yv_0) +\mathbf{s}_{t_0}^T(a\yv_0)+c_{t_0} \\ \leq \frac14(a\yv_0)^TP_{t_0}(a\yv_0)<0,\ \forall a > a_0.
    \end{multline}    
    However, due to the claim~\eqref{eqn:claim1} and the positivity of the terminal cost~\eqref{eq:primary_task}, there exists $a_1>0$     
    such that~$V_{t_0}(\xv)\geq 0$ for any~$\norm{\xv}\geq a_1$ (cf. definition~\eqref{eq:def_value_func}). Choosing~$a\geq \max{\left\{a_0,\frac{a_1}{\norm{y_0}}\right\}}$ leads to a contradiction regarding the sign of~\(V_{t_0}(a\yv_0)\), which proves~$P_t\succeq0$ for any~$t\in[T]$.

    Lastly, we show the global existence and uniqueness of the solution to~\eqref{eq:backward_recursion}.
    Since~$P_t\succeq0$ for any~$t\in[T]$ and \(R\succ 0\),~$H_t=  R+B^TP_{t+1}B$ remains invertible at any time~$t$, which concludes the proof.

\subsection{Proof of Corollary~\ref{cor:two_special_cases}}\label{apdx:pf_two_special_cases}

    Firstly, we claim that when $\lambda=0$ or $\lambda = R_\beta\Sigma_y^2$, the matrix \(P_t\) and the vector \(\mathbf{s}_t\) have the following sparsity patterns 
    \begin{equation}\label{eq:sparcity_patterns}
        P_{t} = \begin{bmatrix}
        \rho_{t}&0\\0&0
    \end{bmatrix}, \quad \mathbf{s}_{t} = \begin{bmatrix}
        \eta_{t}\\0
    \end{bmatrix},\quad \forall t\in[T],
    \end{equation}
    for some $\rho_t, \eta_t \in \mathbb{R}$.
    We prove this claim by induction. Since \(P_T\) and \(\mathbf{s}_T\) exhibit the desired sparsity patterns (cf.~\eqref{eq:terminal_conditions}), it suffices to show that if \(P_{t+1}\) and \(\mathbf{s}_{t+1}\) satisfy the sparsity patterns in~\eqref{eq:sparcity_patterns}, 
    then \(P_t\) and \(\mathbf{s}_{t}\) also satisfy them.
    Indeed, by~\eqref{eq:matrices_for_new_cost} and~\eqref{eq:backward_recursion},
        \begin{align}
        P_t & = M_t - G_tH_t^{-1}G_t^T \\
    & = Q_t + A^TP_{t+1}A  \\ & \quad\quad- (N_t+A^TP_{t+1}B)(R+P_{t+1})^{-1}(N_t+A^TP_{t+1}B)^T \\
    & = \begin{bmatrix}
        R_v + \frac{R_\alpha\rho_{t+1}}{R_\alpha + \rho_{t+1}} & 0\\0& 0
    \end{bmatrix} + \frac{\lambda}{\Sigma_y^2}\Big(1-\frac{\lambda}{R_\beta\Sigma_y^2}\Big)\begin{bmatrix}
         (F_t^b)^2 & F_t^bF_t^c\\F_t^bF_t^c& (F_t^c)^2
     \end{bmatrix},
    \end{align}
    \vspace{-2em}
    \begin{align}
        \mathbf{s}_t &= \mathbf{q}_t + A^T\mathbf{s}_{t+1} - G_tH_t^{-1}(\mathbf{r}_t + B^T \mathbf{s}_{t+1})\\
        & = \mathbf{q}_t + A^T\mathbf{s}_{t+1} - (N_t+A^TP_{t+1}B)(R+P_{t+1})^{-1}(\mathbf{r}_t + B^T\mathbf{s}_{t+1}) \\
        & = \begin{bmatrix}
            -R_v\bar{\mathbf{v}}_t +\frac{R_\alpha\eta_{t+1}}{R_\alpha + \rho_{t+1}} \\ 0
        \end{bmatrix} + \frac{\lambda}{\Sigma_y^2}\Big(1-\frac{\lambda}{R_\beta\Sigma_y^2}\Big)\begin{bmatrix}
            F_t^bF_t^d\\ F_t^cF_t^d
        \end{bmatrix},
    \end{align}
    which proves the claim.  
    
    When $P_t$ and $\mathbf{s}_t$ satisfy the sparsity patterns in~\eqref{eq:sparcity_patterns}, by~\eqref{eq:optimal_control}, the optimal control is given by
    \begin{align}
        \mathbf{u}_t^* &= -H_t^{-1}(G_t^T\mathbf{x}_t +\mathbf{r}_t + B^T\mathbf{s_{t+1})} \\
        & = -(R+P_{t+1})^{-1}((N_t + A^TP_{t+1})^T\mathbf{x}_t+\mathbf{r}_t + \mathbf{s}_{t+1})\\
        & = -\begin{bmatrix}
            \frac{1}{R_\alpha+\rho_{t+1}} & 0 \\ 0 & \frac{1}{R_\beta}
        \end{bmatrix}\left(\begin{bmatrix}
            \rho_{t+1} & 0 \\ -\frac{\lambda F_t^b}{\Sigma_y^2} & -\frac{\lambda F_t^c}{\Sigma_y^2} 
        \end{bmatrix}\mathbf{x}_t + \begin{bmatrix}
            \eta_{t+1} \\ -\frac{\lambda \mathbf{f}_t^d}{\Sigma_y^2}
        \end{bmatrix}\right)\\
        & = \begin{bmatrix}
            -\frac{\rho_{t+1}\mathbf{v}_t+\eta_{t+1}}{R_\alpha+\rho_{t+1}} \\ \frac{\lambda (F_t^b \mathbf{v}_t + F_t^c \mathbf{p}_t + \mathbf{f}_t^d)}{R_\beta \Sigma_y^2}
        \end{bmatrix}.
    \end{align}
    It is clear that if \(\lambda = 0\), then $\beta^* \equiv 0$; if~$\lambda = R_\beta\Sigma_y^2$, then $\beta_t^* = F_t^b \mathbf{v}_t + F_t^c \mathbf{p}_t + \mathbf{f}_t^d$, with the optimal~$\alpha^*$ to be identical in both cases.
    This concludes the proof.

\subsection{Proof of Lemma~\ref{prop:E_log_L_t_star}}\label{apdx:pf_E_log_L_t_star}
    
    From the state dynamics~\eqref{eq:matrix_form_state_dynamics}, the optimal control~\eqref{eq:optimal_control}, and Assumption~\ref{assu:sd_x0}, 
    \begin{align}\label{eq:x_star}
        \mathbf{x}^*_k &= K_{k-1}^T\mathbf{x}^*_{k-1}- H^{-1}_{k-1}\mathbf{s}_k + \mathbf{w}_{k-1}.
    \end{align}
    Since~$\xv_0^*$ is deterministic and~$\{\wv_t\}_{t\in[T-1]}$ are i.i.d. Gaussian,~$\xv_k^*$ is Gaussian for any~$k\in[T]$.
    Taking expectations on both sides yields~\eqref{eq:mean_x_k}.
    Since~$R$,~$R_v$, and~$T_v$ are symmetric, by~\eqref{eq:backward_recursion} and~\eqref{eq:terminal_conditions},~$P_t$ is symmetric, hence~$H_t^{-1}$ is symmetric.
    Plugging~\eqref{eq:x_star} into~\(\Sigma^*_k = \E [\mathbf{x}^*_k(\mathbf{x}^*_k)^T] - m^*_k(m^*_k)^T\) yields~\eqref{eq:cov_x_k}.
    This proves the first part of the Lemma.

    Next, we compute $\E\log L^*_t$. Using~\eqref{eq:state_dynamics},~\eqref{eq:log_L_T} and Assumption~\ref{assu:sd_x0},
    \begin{align}\label{eq:log_L_t_star}
        \log L_t^* &= \frac12\sum_{k=0}^{t-1} \Big\{\norm{\beta^*_k + Y_k}^2_{(\Sigma_y^2)^{-1}}\\&\qquad- \norm{\beta^*_k + Y_k - F^b_k \mathbf{v}^*_k -F^c_{k} \mathbf{p}^*_{k}}^2_{(\Sigma_y^2)^{-1}}\Big\}\\
        \label{eq:log_L_t_star_matrix}
        &= \sum_{k=0}^{t-1}\Big[(\beta^*_k + Y_k)^T(\Sigma_y^2)^{-1}F_k^T \mathbf{x}^*_k  - \frac12(\mathbf{x}^*_k)^TF_k(\Sigma_y^2)^{-1}F_k^T\mathbf{x}^*_k\Big].
    \end{align}
    Noticing~$\beta^*_k = E_2^T\mathbf{u}^*_k$ and plugging~\eqref{eq:optimal_control} into the equation above yields
    \begin{multline}
        \log L_t^*  = \sum_{k=0}^{t-1}\Big[Y_k^T(\Sigma_y^2)^{-1}F_k^T \mathbf{x}^*_k - \mathbf{s}_{k+1}^TH_k^{-1}E_2(\Sigma_y^2)^{-1}F_k^T\mathbf{x}^*_k \\
        - \frac12(\mathbf{x}^*_k)^T(F_k(\Sigma_y^2)^{-1}F_k^T + 2G_kH_k^{-1}E_2(\Sigma_y^2)^{-1}F_k^T)\mathbf{x}^*_k\Big].
    \end{multline}
    Since the randomness comes from two independent Gaussian sources, namely \(Y_k\) and \(\mathbf{x}_k^*\),  
    \begin{multline}
        \mathbb{E}\log L_t^* = -\frac12\sum_{k=0}^{t-1}\Big( 2\mathbf{s}_{k+1}^TH_k^{-1}E_2(\Sigma_y^2)^{-1}F_k^T \mathbb{E}[\mathbf{x}^*_k]  \\+ \mathbb{E}[(\mathbf{x}^*_k)^T(F_k(\Sigma_y^2)^{-1}F_k^T + 2G_kH_k^{-1}E_2(\Sigma_y^2)^{-1}F_k^T)\mathbf{x}^*_k]\Big).
    \end{multline}
    Using \(\mathbf{x}^*_k\sim \mathcal{N}(m^*_k,\Sigma^*_k)\),  
    \begin{multline}
        \mathbb{E}[(\mathbf{x}^*_k)^T(F_k(\Sigma_y^2)^{-1}F_k^T + 2G_kH_k^{-1}E_2(\Sigma_y^2)^{-1}F_k^T)\mathbf{x}^*_k] \\ =  \mathrm{Tr}((F_k(\Sigma_y^2)^{-1}F_k^T + 2G_kH_k^{-1}E_2(\Sigma_y^2)^{-1}F_k^T)\E[\mathbf{x}^*_k(\mathbf{x}^*_k)^T]),
    \end{multline}
    which proves~\eqref{eqn:E_log_L_star}.
    This concludes the proof.

\subsection{Proof of Theorem~\ref{thm:probabilistic_bound}}\label{apdx:pf_probabilistic_bound}

    Firstly, we claim that
    \begin{equation}\label{eq:log_L_t_star_quadratic}
        \log L^*_t = \mathcal{L}^0_t + (\mathcal{L}^1_t)^T\mathbf{w}^{(t)} + (\mathbf{w}^{(t)})^T\mathcal{L}^{2}_t\mathbf{w}^{(t)}.
    \end{equation} 
    where \(\mathbf{w}^{(t)} := \mathrm{Concat}(\xv_0, \mathbf{w}_0, \ldots, \mathbf{w}_{t-1})\) is an \(\R^{2n(t+1)}\)-valued random vector and $\mathcal{L}_t^0\in \R$.    
    
    To see why this claim holds, we represent $\xv^*_k$ in terms of $\{\wv_j\}_{j=-1}^{t-1}$, adopting the convention that $\mathbf{w}_{-1} := \mathbf{x}_0$. 
    By recursively applying~\eqref{eq:x_star}, we get
    \begin{equation}
        \mathbf{x}^*_k = \prod_{j=0}^{k-1}K_{j}^T\mathbf{x}_0 - \sum_{j=0}^{k-1}\Big(\prod_{i=j+1}^{k-1}K_{i}^T\Big)H_j^{-1}\mathbf{s}_{j+1}+ \sum_{j=0}^{k-1}\Big(\prod_{i=j+1}^{k-1}K_{i}^T\Big)\mathbf{w}_j,
    \end{equation}
    where the matrix product \(\prod_{i=j+1}^{k-1}K_{i}^T := K_{k-1}^T\ldots K_{j+1}^T\) is defined in reverse order to account for the non-commutativity of matrix multiplication.
    The representation simplifies to
    \begin{equation}\label{eq:x_star_in_w}
        \xv_k^* = \mathcal{D}_{k-1} + \sum_{j=-1}^{k-1}\mathcal{C}_{j+1,k-1}\mathbf{w}_j.
    \end{equation}
    Substituting~\eqref{eq:x_star_in_w} and~$Y_k = E_2^T\wv_k$ into~\eqref{eq:log_L_t_star_matrix} gives
    \begin{multline}
        \log L_t^* =  \sum_{k=0}^{t-1}\Big\{-\mathbf{s}_{k+1}^TH_k^{-1}E_2(\Sigma_y^2)^{-1}F_k^T\mathcal{D}_{k-1}\\ - \frac12\mathcal{D}_{k-1}^T\mathcal{S}_k\mathcal{D}_{k-1}
    +\mathcal{D}_{k-1}^TF_k(\Sigma_y^2)^{-1}E_2^T\wv_k\\
    + \sum_{j=-1}^{k-1}\mathbf{w}_k^TE_2(\Sigma_y^2)^{-1}F_k^T\mathcal{C}_{j+1,k-1}\mathbf{w}_j \\
    - \sum_{j=-1}^{k-1}[\mathbf{s}_{k+1}^TH_k^{-1}E_2(\Sigma_y^2)^{-1}F_k^T + \mathcal{D}_{k-1}^T\mathcal{S}_k^s ]\mathcal{C}_{j+1,k-1}\mathbf{w}_j \\ - \frac12\sum_{i,j=-1}^{k-1}\mathbf{w}_i^T \mathcal{C}_{i+1,k-1}^T\mathcal{S}_k^T\mathcal{C}_{j+1,k-1}\mathbf{w}_j\Big\}.
    \end{multline}
    Interchanging the order of the summations provides~\eqref{eq:log_L_t_star_quadratic}, where
    \begin{equation}
        \mathcal{L}^0_t := -\frac{1}{2}\sum_{k=0}^{t-1}\Big[2\mathbf{s}_{k+1}^TH_k^{-1}E_2(\Sigma_y^2)^{-1}F_k^T\mathcal{D}_{k-1} + \mathcal{D}_{k-1}^T\mathcal{S}_k\mathcal{D}_{k-1}\Big].
    \end{equation}

    Next, we bound $\mathbb{P}(\log L_t^*\geq \log \upperthreshold)$ based on the representation~\eqref{eq:log_L_t_star_quadratic}.
    Clearly, $m^w_t = \E [\mathbf{w}^{(t)}]$ and~$\Sigma^w_t = \mathrm{cov}(\mathbf{w}^{(t)})$.
    Let~$\bar{\wv}^{(t)} := \wv^{(t)} - m_t^w$ be the centered random vector. By~\eqref{eq:log_L_t_star_quadratic},
    \begin{align}\label{eq:log_L_t_star_decomposed}
        \log L^*_t 
        &= \mathcal{L}^0_t + (\mathcal{L}^1_t)^T\mathbf{w}^{(t)} + (\mathbf{w}^{(t)})^T\mathcal{L}^{2,s}_t\mathbf{w}^{(t)}\\
        &= [\mathcal{L}^0_t + (\mathcal{L}^1_t)^Tm^w_t +(m^w_t)^T\mathcal{L}^{2,s}_tm^w_t]\\ & \qquad + (\mathcal{L}^1_t + 2\mathcal{L}^{2,s}_tm^w_t)^T\bar{\wv}^{(t)} + (\bar{\wv}^{(t)})^T\mathcal{L}^{2,s}_t\bar{\wv}^{(t)}.
    \end{align}
    
    Using Lemma~\ref{lem:HW}, we bound the quadratic form in $\bar{\wv}^{(t)}$:
    \begin{multline}\label{eq:bound_quadratic_part}\mathbb{P}\Big((\bar{\wv}^{(t)})^T\mathcal{L}^{2,s}_t\bar{\wv}^{(t)} - \E[(\bar{\wv}^{(t)})^T\mathcal{L}^{2,s}_t\bar{\wv}^{(t)}]\geq \EPS_2\Big)\leq \\
        \mathrm{exp}\Big\{-\frac18\Big(\frac{\EPS_2^2}{\|(\Sigma^w_t)^{\frac12}\mathcal{L}^{2,s}_t(\Sigma^w_t)^{\frac12}\|_F^2}\wedge \frac{\EPS_2}{\|(\Sigma^w_t)^{\frac12}\mathcal{L}^{2,s}_t(\Sigma^w_t)^{\frac12}\|_2}\Big)\Big\},
    \end{multline}
    for any $\epsilon_2>0$.
    
    Using Lemma~\ref{lem:Gaussian_Concen}, we bound the linear form in $\bar{\wv}^{(t)}$:
    \begin{multline}\label{eq:bound_linear_part}
        \mathbb{P}((\mathcal{L}^1_t + 2\mathcal{L}^{2,s}_tm^w_t)^T\bar{\wv}^{(t)} - \E[(\mathcal{L}^1_t + 2\mathcal{L}^{2,s}_tm^w_t)^T\bar{\wv}^{(t)}]\geq \EPS_1) \\ \leq
        \mathrm{exp}\Big\{-\frac{\EPS_1^2}{2\|(\Sigma^w_t)^{\frac12}(\mathcal{L}^1_t + 2\mathcal{L}^{2,s}_tm^w_t)\|_2^2}\Big\},
    \end{multline}
    which holds for any $\epsilon_1>0$.
    
    By~\eqref{eq:log_L_t_star_decomposed}, combining both bounds~\eqref{eq:bound_quadratic_part} and~\eqref{eq:bound_linear_part} yields
    \begin{multline}
        \mathbb{P}(\log L^*_t - \E \log L^*_t\geq \EPS_1 + \EPS_2) \\\leq \mathrm{exp}\Big\{-\frac{\EPS_1^2}{2\|(\Sigma^w_t)^{\frac12}(\mathcal{L}^1_t + 2\mathcal{L}^{2,s}_tm^w_t)\|_2^2}\Big\} \\
        + \mathrm{exp}\Big\{-\frac18\Big(\frac{\EPS_2^2}{\|(\Sigma^w_t)^{\frac12}\mathcal{L}^{2,s}_t(\Sigma^w_t)^{\frac12}\|_F^2}\wedge \frac{\EPS_2}{\|(\Sigma^w_t)^{\frac12}\mathcal{L}^{2,s}_t(\Sigma^w_t)^{\frac12}\|_2}\Big)\Big\}
    \end{multline}
    for any $\epsilon_1, \epsilon_2>0$.
    
    Whenever \( \E\log L^*_t + \EPS_1+\EPS_2\leq \log \upperthreshold\) for some \(\EPS_1,\EPS_2>0\),
    \begin{equation}
        \mathbb{P}(\log L^*_t \geq \log \upperthreshold)\leq \mathbb{P}(\log L^*_t - \E \log L^*_t\geq \EPS_1 + \EPS_2),
    \end{equation}
    which concludes the proof.
    Since the probabilistic upper bound is monotone decreasing in~\(\EPS_1\) and~\(\EPS_2\), the bound is minimized (yielding the tightest upper bound) when~\(\EPS_1\) and~\(\EPS_2\) take maximum values, i.e., \(\E\log L^*_t + \EPS_1+\EPS_2 = \log \upperthreshold\).